\documentclass[10pt, leqno]{article}

\usepackage{mathrsfs}
\usepackage{amsmath}
\usepackage{mathtools}
\usepackage{amssymb}
\usepackage{amsthm}
\usepackage{comment}
\usepackage[dvipdfmx]{hyperref}

\usepackage{accents}

\newtheorem{definition}{Definition}[section]
\newtheorem{them}{Theorem}[section]
\newtheorem{prop}{Proposition}[section]
\newtheorem{lem}{Lemma}[section]
\newtheorem{cor}{Corollary}[section]

\numberwithin{equation}{section}

\begin{document}

\title{
A direct energy estimates for effectively hyperbolic operators}

\author{Tatsuo Nishitani\footnote{Department of Mathematics, Osaka University:  
nishitani@math.sci.osaka-u.ac.jp
}}

\date{}
\maketitle

\def\ep{\epsilon}
\def\dif{\partial}
\def\al{\alpha}
\def\be{\beta}
\def\ga{\gamma}
\def\om{\omega}
\def\lam{\lambda}
\def\varep{\varepsilon}
\def\R{{\mathbb R}}
\def\N{{\mathbb N}}
\def\C{{\mathbb C}}
\def\Q{{\mathbb Q}}
\def\Ga{\Gamma}
\def\La{\Lambda}
\def\lr#1{\langle{#1}\rangle_{\gamma }}
\def\mD{\lr{ D}_{\mu}}
\def\xim{\lr{\xi}_{\mu}}
\def\co{{\mathcal C}}
\def\op#1{{\rm op}({#1})}
\def\olr#1{\langle{#1}\rangle}
\def\bg{{\bar g}}

\begin{abstract}
This paper is devoted to a simpler derivation of energy estimates, compared to previously existing ones, for effectively hyperbolic operators. One of main points is no use of general Fourier integral operators and another point is an efficient use of the Weyl calculus of pseudodifferential operators associated with several different metrics. 
\end{abstract}

\section{Introduction}

Consider
\begin{equation}
\label{eq:moto:op}
P=-D_t^2+A_2(t, x, D)+A_0(t, x, D)D_t+A_1(t, x, D)
\end{equation}
where $A_j(t, x, D)$ are classical pseudodifferential operators of order $j$ on $\R^d$ depending smoothly on $t$. Denote  the  principal symbol of $P$ by
\[
p(t, x, \tau, \xi )=-\tau^2+a(t, x, \xi )
\]
where $a(t, x, \xi)$ is positively homogeneous of degree $2$ in $\xi$ which is assumed to be  nonnegative for any  $(t, x, \xi)\in U\times \R^d$ with some neighborhood $U$ of $(0, 0)\in \R^{d+1}$, a necessary condition for the Cauchy problem for $P$ to be $C^{\infty}$ well-posed near the origin.

In \cite{IP}, Ivrii and Petkov proved that if the Cauchy problem for $P$  is $C^{\infty}$ well-posed for any lower order term then the Hamilton map $F_p$ has a pair of  non-zero real eigenvalues at every singular point of $p=0$  (\cite[Theorem 3]{IP}). A singular point of $p=0$ is called {\it effectively hyperbolic} (\cite{Ho1})   if the Hamilton map has a pair of non-zero real eigenvalues there. In \cite{I}, Ivrii has proved that if  every singular point is effectively hyperbolic, and  $p$ admits a factorization $p=q_1q_2$ nearby with real smooth symbols  $q_i$, then the Cauchy problem  is $C^{\infty}$ well-posed for every lower order term, reducing $P$ to another with controllable lower order terms, by operator powers of operator.

If a singular point $(t, x, \tau, \xi)$  is effectively hyperbolic then $\tau$ is a characteristic root of multiplicity at most $3$ (\cite[Lemma 8.1]{IP}) and every multiple characteristic root is at most double, the conjecture has been proved in \cite{Iwa1, Iwa2}, \cite{Ni2}. In \cite{Iwa1, Iwa2} the idea of the proof is to reduce $P$ to a perturbation of that treated in \cite{I}  by operator powers of operator of which symbol is found applying the Nash-Moser implicit function theorem. On the other hand in \cite{Ni2} (see also \cite{Ni:book}) the proof is based on weighted energy estimates with pseudodifferential weights of which symbol stems from a geometric characterization of effectively hyperbolic singular points, after some preliminary transformations by Fourier integral operators. For the Cauchy problem for operators with triple effectively hyperbolic characteristics, where  $p$ cannot be smoothly factorized, see \cite{Ni:Ivconj} and the references given there. 

In this paper, though  we follow \cite{Ni:book} mainly, we derive energy estimates using only  changes of local coordinates $x$ and  the Weyl calculus of pseudodifferential operators, which makes much simpler the arguments to derive local existence of solution to the Cauchy problem (Theorem \ref{thm:sonzai:ippan} below) from microlocal energy estimates. On the other hand, in \cite{Ni:Pisa} we gave another way to obtain microlocal energy estimates without use of Fourier integral operators where, in spite of $C^{\infty}$ problem, we need a calculus of Gevrey pseudodifferential operators in the $(t, x)$-space and a technical verification of  support of solutions.

In Section \ref{sec:direct:ene} we derive (microlocal) weighted energy estimates  and prove local existence result for the Cauchy problem. In Section \ref{sec:jyunbi} several lemmas and propositions required in Section \ref{sec:direct:ene} are stated without proofs, of which proofs are given in Sections \ref{sec:proof:lemma}. In the last section \ref{sec:furoku} we give a proof  of Proposition \ref{pro:p:to:psi} below  for the sake of completeness. 

\section{Preparations for direct energy estimates}
\label{sec:jyunbi}

First recall \cite[Lemmas 3.1, 3.2]{Ni:katata} (see also \cite[Section 2.1]{Ni:book}).
\begin{prop}
\label{pro:p:to:psi}Assume that $(0, 0, 0, {\bar\xi})$ is effectively hyperbolic. One can choose a local coordinates $x$ with ${\bar\xi}=e_d$ and smooth function $\psi(x, \xi)$ such that either $d\psi=d\xi_1$ or $d\psi=\varep dx_1+cdx_d$ at $(0, e_d)$ where $c\,\in\R$, $\varep=0$ or $1$, and smooth $\ell(t, x, \xi)$, $q(t, x, \xi)\geq 0$ vanishing at $( 0, e_d)$, positively homogeneous in $\xi$ of degree $1$, $2$ respectively   
such that
\[
p(t, x, \tau, \xi)=-\tau^2+\ell^2(t, x, \xi)+q(t, x, \xi),\quad q(t, x, \xi)\geq c(t-\psi)^2|\xi|^2
\]
with some $c>0$ on a conic neighborhood of $(0, e_d)$ where
\begin{equation}
\label{eq:kap:okisa}
|\{\ell, \psi\}(0, e_d)|<1,\quad \{\psi, \{\psi, q\}\}(0, e_d)=0.
\end{equation}
\end{prop}
Note that the change of coordinates  can be extended to a diffeomorphism on $\R^d$ which is a linear transformation outside a neighborhood of $x=0$. 
According to $d\psi=d\xi_1$ or $d\psi=\varep x_1+cx_d$ at $(0, e_d)$ one can write
\begin{equation}
\label{eq:psi:gutai}
\psi(x, \xi)=\xi_1/|\xi|+r(x, \xi),\quad
\psi(x, \xi)=\varep x_1+cx_d+r(x, \xi)
\end{equation}
where $dr(0, e_d)=0$. Note that $\{\psi, \{\psi, q\}\}(0, e_d)=0$ implies that
\begin{equation}
\label{eq:kakkosiki}
\dif_{x_1}^2q(0, e_d)=0\;\;\text{if}\;\;d\psi=d\xi_1,\quad\dif_{\xi_1}^2q(0, e_d)=0\;\;\text{if}\;\;d\psi=\varep dx_1+cdx_d.
\end{equation}
We call $(a)$ the coordinates change which leads to $d\psi=d\xi_1$ and call $(b)$ which leads to $d\psi=\varep dx_1+cdx_d $. 

\subsection{Localization of symbols}
\label{sec:kakutyo}

After making a change of coordinates in Proposition \ref{pro:p:to:psi} we localize such obtained symbol (operator) to a neighborhood of $(0, e_d)$. 
We first localize coordinates functions. Let  $\chi(s)\in C^{\infty}(\R)$ be equal to $s$ on $|s|\leq 1$, $|\chi(s)|$ is constant for $|s|\geq 2$ and $0\leq d\chi(s)/ds=\chi^{(1)}(s)\leq 1$ everywhere. Define $y(x)=(y_1(x),\ldots, y_d(x))$ and ${\eta}(\xi)=(\eta_1(\xi),\ldots, \eta_d(\xi))$ by
\begin{align*}
y_j(x)=M^{-1}\chi(M x_j),\;\;{\eta}_j(\xi)=M^{-1}\chi(M(\xi_j\lr{\xi}^{-1}-\delta_{jd}))
\end{align*}
for $j=1,2,\ldots, d$ with $\lr{\xi}=(\gamma^2+|\xi|^2)^{1/2}$ 
 where $\delta_{i j}$ is the Kronecker's delta and  $M$, $\gamma$ are large positive parameters constrained 
\begin{equation}
\label{eq:seigen}
\gamma\geq M^{4}.
\end{equation}
It is easy to see that  $(1-CM^{-1})\lr{\xi}\leq |(\eta+e_d)\lr{\xi}|\leq (1+CM^{-1})\lr{\xi}$ and
\begin{equation}
\label{eq:y:atai}
|y(x)|\leq CM^{-1},\quad |\eta(\xi)|\leq CM^{-1},\quad (x,\xi)\in\R^d\times\R^d
\end{equation}
with some $C>0$ so that $(y(x), \eta(\xi)+e_d)$ is contained in a neighborhood of $(0, e_d)$, shrinking with $M$. Note that $(y, (\eta +e_d)\lr{\xi})=( x,\xi)$  on the conic neighborhood $W_{M,\gamma}$ of $(0, e_d)$;
\begin{equation}
\label{eq:conic:nbd} 
W_{M, \gamma}=\big\{(x,\xi)\mid |x|\leq M^{-1},\; |\xi/|\xi|-e_d|\leq M^{-1}/2,\; |\xi|\geq \gamma M^{1/2} \big\}
\end{equation}
because  if $(x,\xi)\in W_{M, \gamma}$ then
\begin{align*}
\big|\xi/\lr{\xi}-e_d\big|\leq \big|\xi/\lr{\xi}-\xi/|\xi|\big|+\big|\xi/|\xi|-e_d\big|
\leq M^{-1}/2\\+\big(\big|\lr{\xi}-|\xi|\big|\big)/\lr{\xi}
\leq M^{-1}/2+\gamma^2\lr{\xi}^{-1}(\lr{\xi}+|\xi|)^{-1}\leq M^{-1}.
\end{align*}
In what follows we assume that the range of $t$ is constrained such that 
\begin{equation}
\label{eq:t:seigen}
|t|<T_0M^{-1}=\delta
\end{equation}
with some fixed $T_0>0$. 

Let $f(x, \xi)\in S^l_{1, 0}(W)$ where $W$ is a conic neighborhood of $(0, e_d)$. We define the localization of $f$ by $f_M(x, \xi)=f(y(x), (\eta(\xi)+e_d)\lr{\xi})$   which  depends also on $\gamma$ and coincides with the original $f$ in $W_{M,\gamma}$ if $M$ is large. Denote the coordinates change in Proposition \ref{pro:p:to:psi}, extended to $\R^d$, by $x\mapsto \kappa(x)$ and $(Tu)(t, x)=u(t, \kappa(x))$ then the localized symbol of $T^{-1}PT$ is written as
\[
-\tau^2+\ell_M^2(t, x,\xi)+q_M(t, x, \xi)+b_1(t, x,\xi)+b_0(t, x,\xi)\tau
\]
which we denote by ${\hat P}$ from now on.  
All symbols (operators) with which we work in this paper are obtained making two different coordinates changes in Proposition \ref{pro:p:to:psi}. To clarify which coordinates change is employed we write
\[
\text{assertion},\;\;(a)\;\;\;(\text{respectively}\;(b))
\]
which means that the assertion holds when the  coordinates change $(a)$ is chosen (respectively when  $(b)$ is chosen).  If the assertion contains $\ep $  we mean that the assertion corresponding to $\ep$ holds when we choose the coordinates change $(\ep)$, $\ep=a, b$. If the assertion contains neither $(a)$, $(b)$ nor $\ep$, it means that the assertion holds for both coordinates changes $(a)$ and $(b)$.

Let 
\[
G=M^{2}|dx|^2+M^2\lr{\xi}^{-2}|d\xi|^2=M^2\big(|dx|^2+\lr{\xi}^{-2}|d\xi|^2\big).
\]
\begin{lem}
\label{lem:kakutyo:a}Let $f(z)$ be a smooth function in a neighborhood of ${\bar z}$ and let $z_j(x,\xi)\in S(M^{-1}, G)$ and ${f_M}(x,\xi)=f(z(x,\xi)+{\bar z})$. Then ${f_M}(x,\xi)\in S(M^{-r}, G)$
if $\dif_z^{\al}f({\bar z})=0$ for $0\leq |\al|<r$. In particular $
{f_M}(x,\xi)-f({\bar z})\in S(M^{-1},G)$. 
\end{lem}
It is easy to see $
y(x)\in S(M^{-1},G)$ and $\eta(\xi)\in S(M^{-1}, G)$. 
Indeed $y_j\in S(M^{-1}, G)$ is clear while we see
\begin{gather*}
\big|\dif_{\xi}^{\al}\eta_j(\xi)\big|\precsim \sum_{|\al_i|\geq 1} M^{-1} |\chi^{(s)}(M(\xi_j\lr{\xi}^{-1}-\delta_{jd}))|\\
\times|\dif_{\xi}^{\al_1}(M(\xi_j\lr{\xi}^{-1}-\delta_{jd}))|\cdots |\dif_{\xi}^{\al_s}(M(\xi_j\lr{\xi}^{-1}-\delta_{jd}))|\\
\precsim \sum_{s\leq |\al|} M^{-1} M^s\lr{\xi}^{-|\al|}\precsim M^{-1+|\al|}\lr{\xi}^{-|\al|}
\end{gather*}
so that $\eta_j\in S(M^{-1}, G)$ where $A\precsim B$ means  that $A\leq CB$ with some $C>0$ independent of $M$ and $\gamma$.

\begin{lem}
\label{lem:eta:seimitu:a}We have $\dif \eta_j/\dif \xi_k-\delta_{j k}\chi^{(1)}(M\xi_j\lr{\xi}^{-1})\lr{\xi}^{-1}\in S(M^{-1}\lr{\xi}^{-1}, G)$ for $1\leq j\leq d-1$.
\end{lem}
By Lemma \ref{lem:kakutyo:a} we have $\psi_M(x, \xi)=\psi(y(x),\eta(\xi)+e_d)\in S(M^{-1}, G)$ 
which we denote by $\psi(x, \xi)$ dropping $M$ to simplify notation. Denoting 
\begin{equation}
\label{eq:ell:bar}
{\bar\ell}(t, x, \xi)=\ell( t, y(x), \eta(\xi)+e_d),\quad {\bar q}(t, x, \xi)=q( t, y(x), \eta(\xi)+e_d)
\end{equation}
we have $\ell_M={\bar \ell}(t, x, \xi)\lr{\xi}$ and $q_M={\bar q}(t, x, \xi)\lr{\xi}^2$ which we denote by $\ell(t, x, \xi)$ and $q(t, x, \xi)$ dropping $M$ again. Note that 
${\bar \ell}\in S(M^{-1}, G)$ and ${\bar q}\in S(M^{-2}, G)$ in view of Lemma \ref{lem:kakutyo:a} and Proposition \ref{pro:p:to:psi} shows that
\begin{equation}
\label{eq:q:sita:a}
\begin{split}
{\bar q}(t, x, \xi)
\geq c\,(t-\psi(x,\xi))^2.
\end{split}
\end{equation}
\begin{lem}
\label{lem:kihon:1}We have $q\in S(M^{-2}\lr{\xi}^2, G)$. There exists $C>0$ such that
\begin{gather*}
\big|\dif_{x_1}q\big|\leq CM^{-1/2}\sqrt{q}\,\lr{\xi},
\quad \big|\dif_{x_j}q\big|\leq C\sqrt{q}\,\lr{\xi},\;\;j\neq 1,\quad (a),\\
\big|\dif_{\xi_j}q\big|\leq CM^{-1/2}\sqrt{q},\;\;j=1, d,\quad \big|\dif_{\xi_j}q\big|\leq C\sqrt{q},\;\;j\neq 1, d,\quad (b).
\end{gather*}
\end{lem}
\begin{lem}
\label{lem:kihon:2}We have $\psi\in S(M^{-1}, G)$ and
\[
\begin{drcases}
{ \psi}(x, \xi)-M^{-1}\chi(M\xi_1\lr{\xi}^{-1})\in S(M^{-2}, G)\\
 \dif{ \psi}/\dif{\xi_k}-\delta_{1 k}\chi^{(1)}(M\xi_1\lr{\xi}^{-1})\lr{\xi}^{-1}\in S(M^{-1}\lr{\xi}^{-1}, G)
 \end{drcases}\quad(a),
 \]
\[
\begin{drcases}
{\psi}(x, \xi)-\varep M^{-1}\chi(Mx_1)-cM^{-1}\chi(Mx_d)\in S(M^{-2}, G),\\
\dif{\psi}/\dif{x_k}-\varep\delta_{1 k}\chi^{(1)}(Mx_1)-c\delta_{d k}\chi^{(1)}(Mx_d)\in S(M^{-1}, G)
\end{drcases}\quad (b).
\]
\end{lem}
\begin{prop}
\label{pro:poison:1}We have $
\big|\{q, {\psi}\}\big|\leq CM^{-1/2}\sqrt{q}\,\lr{\xi}^{-1}$.
\end{prop}
\begin{proof}The proof is clear from Lemmas \ref{lem:kihon:1} and \ref{lem:kihon:2}.
\end{proof}
 Thanks to Lemma \ref{lem:kihon:2}  one sees
\begin{lem}
\label{lem:kokan} We have 
$\{{\ell}, {\psi}\}+\kappa \chi^{(1)}(Mx_1)\chi^{(1)}(M\xi_1\lr{\xi}^{-1})
\in S(M^{-1}, G)$ where $\kappa=\dif_{x_1}\ell(0, e_d)$, $(a)$ or $\kappa=-\dif_{\xi_1}\ell(0, e_d)$, $(b)$and  $|\kappa|<1$ by \eqref{eq:kap:okisa}.
\end{lem}
%

%
\subsection{Approximate square roots and pseudodifferential weights}
\label{sec:square:root}

Introducing a parameter $\lam\geq 1$ we denote
\[
{\bar b}=\big({\bar q}+\lam \lr{\xi}^{-1}\big)^{1/2}
\]
so that $b=\lr{\xi}{\bar b}=(q+\lam \lr{\xi})^{1/2}$ where $\lam$ is  constrained
\begin{equation}
\label{eq:para:seigen}
\lam \leq \gamma M^{-2}, \qquad \lam\geq 1
\end{equation}
such that $\lam \lr{\xi}^{-1}\leq M^{-2}$.   In the end of this section $\lam $ will be fixed.
Introducing 
\begin{equation}
\label{eq:ome:teigi}
\omega=((t-\psi)^2+\lr{\xi}^{-1})^{1/2}
\end{equation}
where $\lr{\xi}^{-1/2}\leq \omega\leq CM^{-1}$ and taking \eqref{eq:q:sita:a} into account one has
\begin{equation}
\label{eq:b:sita}
\begin{split}
b=\big(q+\lam \lr{\xi}\big)^{1/2}\geq \big(c(t-\psi)^2\lr{\xi}^2+\lam \lr{\xi}\big)^{1/2}\\
\geq c\, \omega^{-1}\lr{\xi}\big((t-\psi)^2\omega^2+\omega^2\lr{\xi}^{-1}\big)^{1/2}\\
\geq c\, \omega^{-1}\lr{\xi}\big(|t-\psi|^4+\lr{\xi}^{-2}\big)^{1/2}\geq (c /\sqrt{2})\lr{\xi}\omega
\end{split}
\end{equation}
because $\omega^2\geq \lr{\xi}^{-1}$. Introduce the metric
\[
\bg=\lr{\xi}|dx|^2+\lr{\xi}^{-1}|d\xi|^2
\]
which is one of basic metrics with which we work. Note that  $
S(m, G)\subset S(m, \bg)$ 
because $M^{|\al+\be|}\lr{\xi}^{-|\be|}\leq (M^2\lr{\xi}^{-1})^{|\al+\be|/2}\lr{\xi}^{(|\al|-|\be|)/2}$ 
and $M^{2}\lr{\xi}^{-1}\leq 1$. The metric $\bg$ is $\sigma$ temperate  {\it uniformly} in $\gamma\geq M^4\geq 1$ which will be checked later.
\begin{lem}
\label{lem:b:lam:1}We have ${\bar b}\in S({\bar b}, \bg)$ and $
\dif_x^{\al}\dif_{\xi}^{\be}\,{\bar b} \in S(\lam^{-1/2}\lr{\xi}^{(|\al|-|\be|)/2}\,{\bar b}, {\bg})$ for $|\al+\be|=1$.
\end{lem}
%


From this lemma it follows easily 
\begin{lem}
\label{lem:b:lam:-1}We have $
\dif_x^{\al}\dif_{\xi}^{\be}\,{\bar b}^{-1} \in S(\lam^{-1/2}\lr{\xi}^{(|\al|-|\be|)/2}\,{\bar b}^{-1},  \bg)$ for $|\al+\be|=1$.
\end{lem}
\begin{lem}
\label{lem:b:2}We have $\dif_x^{\al}\dif_{\xi}^{\be}\,{\bar b}\in S(\lr{\xi}^{-|\be|}, \bg)$ for $|\al+\be|=1$.
\end{lem}
\begin{prop}
\label{pro:lam:1}${ b}$ is an admissible weight for ${\bg}$ and ${ b}\in S({ b},  { \bg})$. 
\end{prop}
%

Since $b$ and $b^{-1}$ are admissible weights for $\bg$  we have
\[
b\#b^{-1}=1-r
\]
where $r\in S(\lam^{-1}, \bg)$ which follows from Lemmas \ref{lem:b:lam:1} and \ref{lem:b:lam:-1}. Therefore choosing $\lam\geq 1$ suitably large we have $\|\op{r}\|_{{\mathcal L}(L^2, L^2)}<1$ so that $(I-\op{r})^{-1}$ exists which is given by $(I-\op{r})^{-1}=\op{{\tilde r}}$ with ${\tilde r}\in S(1, \bg)$ (see \cite{Be}, \cite{Ku}). Thus we have $b\#(b^{-1}\#{\tilde r})=1$ and $(b^{-1}\#{\tilde r})\#b=1$ where ${\tilde b}=b^{-1}\#{\tilde r}\in S(b^{-1}, \bg)$. 
 We summarize
\begin{prop}
\label{pro:b:lam}One can find $\lam\geq 1$ independent of $M$ and $\gamma$  such that there exists ${\tilde b}\in S(b^{-1}, \bg)$ satisfying $b\#{\tilde b}={\tilde b}\#b=1$.
\end{prop}
From now on we {\it fix} such a $\lambda={\bar \lam}$ 
while $M$ and $\gamma$ remain to be free with the constraints \eqref{eq:seigen} and \eqref{eq:para:seigen}.
 \begin{lem}
 \label{lem:dif:q} We have ${\bar q}\in S(\lr{\xi}^{-1/2}\,{\bar b}, \bg)$. Moreover $\dif_{x_1}{\bar q}\in S(M^{-1/2}\,{\bar b}, \bg)$, $(a)$ and $\dif_{\xi_j}{\bar q}\in S(M^{-1/2}\lr{\xi}^{-1}\,{\bar b}, \bg)$ for $j=1, d$, $(b)$.
 \end{lem}
\begin{cor}
\label{cor:lam:2}We have $\dif_{x_1}{\bar b}\in S(M^{-1/2}, \bg)$, $(a)$ and  $\dif_{\xi_j}{\bar b}\in S(M^{-1/2}\lr{\xi}^{-1}, \bg)$ for $j=1, d$, $(b)$.
\end{cor}
\begin{cor}
\label{cor:b:tbibun}We have $\dif_t{\bar b}\in S(1, \bg)$. Moreover $\dif_{x_1}\dif_t{\bar b}\in S(M^{-1/2}\lr{\xi}^{1/2}, \bg)$, $(a)$ and $\dif_{\xi_j}\dif_t{\bar b}\in S(M^{-1/2}\lr{\xi}^{-1/2}, \bg)$ for $j=1, d$, $(b)$.
\end{cor}
 Define $\phi$, the symbol of weight for energy estimates, by
\[
\phi=\sqrt{(t-\psi)^2+\lr{\xi}^{-1}}+t-\psi=\omega+t-\psi
\]
and note that
\begin{equation}
\label{eq:phi:uesita}
M\lr{\xi}^{-1}/C\leq \lr{\xi}^{-1}/(2\omega)\leq \phi\leq CM^{-1}.
\end{equation}
Introduce two metrics $g_{\ep}$, $\ep=a, b$ associated to the case $(a)$ and $(b)$;
\begin{equation}
\label{eq:metric:ab}
g_{\ep}=M^{-2\delta_{\ep a}}\lr{\xi}|dx|^2+M^{-2\delta_{\ep b}}\lr{\xi}^{-1}|d\xi|^2
\end{equation}
where $\delta_{\ep \ep'}=1$ if $\ep=\ep'$ and $\delta_{\ep \ep'}=0$ otherwise. 
The metric $g_{\ep}$ is $\sigma$ temperate  {\it uniformly} in $\gamma\geq M^2\geq 1$ which is checked later. Note that     
\[
M^{|\al+\be|}\lr{\xi}^{-|\be|}\leq (M^{4}\lr{\xi}^{-1})^{|\al+\be|/2}M^{-\ep(\al,\be)}\lr{\xi}^{(|\al|-|\be|)/2}.
\]
so that $S(m, G)\subset S(m, g_{\ep})$ where 
\[
\ep(\al, \be)=|\al|\delta_{\ep a}+|\be|\delta_{\ep b}.
\]
\begin{prop}
\label{pro:dif:omephi:matome}We  have $\omega^s\in S(\omega^s, g_{\ep})$ and $\phi^s\in S(\phi^s, g_{\ep})$. Moreover
\begin{gather*}
\dif_x^{\al}\dif_{\xi}^{\be}\omega^s\in S(M^{-\ep(\al,\be)}\omega^{-1}\lr{\xi}^{-1/2}\lr{\xi}^{(|\al|-|\be|)/2}\omega^s, g_{\ep}),\\
\dif_x^{\al}\dif_{\xi}^{\be}\phi^s\in S(M^{-\ep(\al,\be)}\omega^{-1}\lr{\xi}^{-1/2}\lr{\xi}^{(|\al|-|\be|)/2}\phi^s, g_{\ep})
\end{gather*}
for $|\al+\be|\geq 1$.
\end{prop}
\begin{lem}
\label{lem:phi:cho:seimitu}We have $
\dif_{\xi_j}\phi\in S(M^{-1}\omega^{-1}\lr{\xi}^{-1}\phi, g_a)$, $\dif_{\xi_j}\omega^s\in S(M^{-1}\omega^{s-1}\lr{\xi}^{-1}, g_a)$  for $j\neq 1$, $(a)$ and $\dif_{x_j}\phi\in S(M^{-1}\omega^{-1}\lr{\xi}^{-1}\phi, g_b)$, $\dif_{x_j}\omega^s\in S(M^{-1}\omega^{s-1}\lr{\xi}^{-1}, g_b)$ for $j\neq 1, d$, $(b)$.
\end{lem}
\begin{prop}
\label{pro:omephi:to:g}$\omega$, $\phi$ are admissible weights for both $g_{\ep}$ and $\bg$.
\end{prop}
%

\subsection{Some bounds of pseudodifferential operators}
\label{sec:pdo:bound}

We start with
\begin{lem}
\label{lem:gyaku}
Let $m$ be admissible for $g_{\ep}$ and $p\in S(m, g_{\ep})$ satisfy
$p\geq c\,m$ with some constant $c>0$. Then $p^{-1}\in S(m^{-1},g_{\ep})$ and there exist $k, {\tilde k}\in S(M^{-1},g_{\ep})$ such that
\begin{align*}
p\#p^{-1}\#(1+k)=1,\;\;(1+k)\#p\#p^{-1}=1,\;\;p^{-1}\#(1+k)\#p=1,\\
p^{-1}\#p\#(1+{\tilde k})=1,\;\;(1+{\tilde k})\#p^{-1}\#p=1,\;\;p\#(1+{\tilde k})\#p^{-1}=1.
\end{align*}
\end{lem}
\begin{lem}
\label{lem:Ni:hon:1} Let $q\in S(1, g_{\ep})$ satisfy $q\geq c$ with a constant $c$ independent of $M$. Then there is $C>0$ such that 
\[
\big(\op{q}u,u) \geq (c-CM^{-1/2})\|u\|^2.
\]
\end{lem}
\begin{lem}
\label{lem:Ni:hon:2} Let $q\in S(1, g_{\ep})$ then there is $C>0$ such that 
\[
\|\op{q}u\| \leq \big(\sup{|q|}+CM^{-1/2}\big)\|u\|.
\]
\end{lem}

\begin{lem}
\label{lem:kihon:fu}Let $m>0$ be  admissible for $g_{\ep}$ and  $m\in S(m, g_{\ep})$.  Then
\[
(\op{m}u, u)\geq (1-CM^{-2})\|\op{\sqrt{m}}u\|^2.
\]
If $q\in S(m, g_{\ep})$ then there is $C>0$ such that
\[
\big|(\op{q}u,u)\big|\leq \big(\sup{\big(|q|/m\big)}+CM^{-1/2}\big)\|\op{\sqrt{m}\,}u\|^2.
\]
\end{lem}
\begin{lem}
\label{lem:m:1:2}Let $m_i>0$ be admissible for $g_{\ep}$ and assume that $m_i\in S(m_i, g_{\ep})$ and $m_2\leq C\, m_1$ with  $C>0$. Then there is $C'>0$ such that
\[
\big\|\op{m_2}u\big\|\leq C'\big\|\op{m_1}u\big\|.
\]
\end{lem}
%

\section{Direct energy estimates}
\label{sec:direct:ene}

\subsection{Direct energy estimate for localized operators}

Let 
 \[
 L=\op{\ell},\quad B=\op{b}. 
 \]
 Since $\ell\in S(M^{-1}\lr{\xi}, G)$ then $\dif_x^{\al}\dif_{\xi}^{\be}\ell\in S(M^{-1+|\al+\be|}\lr{\xi}^{1-|\be|}, g_{\ep})$ for $|\al+\be|=2$ hence 
$\ell\#\ell-\ell^2\in S(M^{2}, g_{\ep})\subset S(M^{-2}\lr{\xi}, g_{\ep})$ 
because of \eqref{eq:seigen}, that is
\begin{equation}
\label{eq:l:to:L}
\op{\ell^2}=L^2+\op{r},\quad r\in S(M^{-2}\lr{\xi}, g_{\ep}).
\end{equation}
On the other hand we have $
b\#b=b^2+{\tilde r}=q+{\bar\lam} \lr{\xi}+{\tilde r}$ 
with ${\tilde r}\in S(\lr{\xi}, \bg)$ thanks to Lemma \ref{lem:b:2}. Thus
\begin{equation}
\label{eq:q:to:B}
\op{q}=B^2-\op{r},\quad r={\bar\lam} \lr{\xi}+{\tilde r}\in S(\lr{\xi}, \bg).
\end{equation}
%
Taking $(D_t-i\theta)e^{-\theta t}=e^{-\theta t}D_t$ where $\theta>0$ into account consider
\[
{\hat P_{\theta}}=-A^2+L^2(t, x, D)+B^2(t, x, D)+B_0(t, x, D)D_t+B_1(t, x, D)
\]
with $A=D_t-i\theta$ where $B_i=\op{{\tilde b_i}}$ and  ${\tilde b_i}\in S(\lr{\xi}^i, \bg)$. Since
\begin{gather*}
{\hat P}(t, x, \tau, \xi)=-\tau^2+\ell^2(t, x,\xi)+q(t, x, \xi)+b_{1}(t, x, \xi)\lr{\xi}+b_{0}(t, x,\xi)\tau\\
=-\tau^2+\ell^2(t, x,\xi)+b^2(t, x, \xi)
+(b_{1}(t, x, \xi)-{\bar \lam})\lr{\xi}+b_{0}(t, x,\xi)\tau
\end{gather*}
where $b_{j}=b_j(t, y(x),\eta(\xi)+e_d)$ hence  ${\tilde b_1}$ contains ${\bar \lam}\lr{\xi}$ {\it but ${\bar\lam}$ has been fixed}. Recall that ${\hat P}(t, x, \tau, \xi)$ coincides with the symbol of $T^{-1}PT$ in $W_{M,\gamma}$. 
\begin{definition} 
\label{dfn:norm}\rm 
We set
\[
\varPhi=\op{\phi^{-n}},\quad \varPhi^{\flat}=\op{\omega^{1/2}\phi^{-n}},\quad \varPhi^{\sharp}=\op{\omega^{-1/2}\phi^{-n}}
\]
here and in what follows to simplify notation the power $n$ is not indicated in $\varPhi$, $\varPhi^{\flat}$, $\varPhi^{\sharp}$ which depends on $n$ of course.
\end{definition}
In this section it is assumed that all constants $c$, ${\hat c}$, ${\bar c}$, $c_i$  are independent of $n$, $M$, $\gamma$ and $\theta$ and every constant $C$, may change from line to line, is independent of $M$, $\gamma$ and $\theta$ while may depend on $n$. 

Assume $K^*=K$ (actually we take $K=L$ or $K=B$) then it is easy to see 
\begin{equation}
\label{eq:ene:id:K}
\begin{split}
2{\mathsf{Im}}(\varPhi K^2u, \varPhi Au)=\dif_t\|\varPhi Ku\|^2+2\theta\|\varPhi Ku\|^2\\
+2{\mathsf{Im}}(\varPhi [A, K]u, \varPhi Ku)+2{\mathsf{Im}}([A, \varPhi]Ku, \varPhi Ku)\\
+2{\mathsf{Im}}([\varPhi, K]Au, \varPhi Ku)+2{\mathsf{Im}}(\varPhi Au, [K, \varPhi]Ku).
\end{split}
\end{equation}
Note that $[A, \varPhi]=in\,\op{\omega^{-1}\phi^{-n}}$ and hence
\[
2{\mathsf{Im}}([A, \varPhi]Ku, \varPhi Ku)=2\,n{\mathsf{Re}}(\op{\omega^{-1}\phi^{-n}}Ku, \op{\phi^{-n}} Ku).
\]
Consider $\op{\phi^{-n}}\op{\omega^{-1} \phi^{-n}}=
\op{\phi^{-n}\#(\omega^{-1} \phi^{-n})}$. Since $\phi^{-n}\#(\omega^{-1} \phi^{-n})=\omega^{-1}\phi^{-2n}+r$ with $r\in S(M^{-1}\omega^{-1}\phi^{-2n}, g_{\ep})$ in view of Proposition \ref{pro:dif:omephi:matome} then thanks to Lemma \ref{lem:kihon:fu} one has $
|(\op{r}u, u)|\leq CM^{-1}\|\varPhi^{\sharp}u\|^2$. 
Thus Lemma \ref{lem:kihon:fu} again gives
\begin{equation}
\label{eq:A:to:K}
2{\mathsf{Im}}([A, \varPhi]Ku, \varPhi Ku)\geq 2n(1-CM^{-1})\|\varPhi^{\sharp}Ku\|^2.
\end{equation}

Note that $
\ell\#\phi^{-n}-\phi^{-n}\#\ell=-i\{\ell, \phi^{-n}\}+r$ with $r\in S(M\phi^{-n}, g_{\ep})$
since $\dif_x^{\al}\dif_{\xi}^{\be}\phi^{-n}\in S(M^{-\ep(\al,\be)}\omega^{-1}\lr{\xi}^{-1/2}\lr{\xi}^{(|\al|-|\be|)/2}\phi^{-n}, g_{\ep})$ for $|\al+\be|=2$ by Proposition \ref{pro:dif:omephi:matome} and 
$\dif_x^{\be}\dif_{\xi}^{\al}\ell\in S(M^{-1+|\be+\al|}\lr{\xi}^{1-|\al|}, g_{\ep})$ for $|\al+\be|=2$ and 
$\omega\geq \lr{\xi}^{-1/2}$. Note that
\[
\{\ell, \phi^{-n}\}=-in\,\omega^{-1}\{\ell, \psi\}\phi^{-n}+in\,\omega^{-1}\{\ell, \lr{\xi}^{-1}\}\phi^{-n-1}
\]
where $\omega^{-1}\{\ell, \lr{\xi}^{-1}\}\phi^{-n-1}\in S(\phi^{-n}, g_{\ep})$ in view of \eqref{eq:phi:uesita}. Therefore we have $
\ell\#\phi^{-n}-\phi^{-n}\#\ell=in\{\ell, \psi\}\omega^{-1}\phi^{-n}+r$ with $ r\in S(M\phi^{-n}, g_{\ep})$. 
Thanks to Proposition \ref{pro:dif:omephi:matome} one has $\phi^{-n}\#(\{\ell, \psi\}\omega^{-1}\phi^{-n})-\{\ell, \psi\}\omega^{-1}\phi^{-2n}\in S(M^{-1}\omega^{-1}\phi^{-2n}, g_{\ep})$ since  $\{\ell, \psi\}\in S(1, g_{\ep})$. Thus one can write
\[
\phi^{-n}\#(\ell\#\phi^{-n}-\phi^{-n}\#\ell)=in\{\ell, \psi\}\,\omega^{-1}\phi^{-2n}+r_1+r_2
\]
where $r_1\in S(M^{-1}\omega^{-1}\phi^{-2n}, g_{\ep})$ and $r_2\in S(M\phi^{-2n}, g_{\ep})$. Write 
\[
(1+k)\#(\omega^{1/2}\phi^n)\#(\{\ell, \psi\}\omega
^{-1}\phi^{-2n})\#(\omega^{1/2}\phi^n)\#(1+{\tilde k})=r
\]
 with $k, {\tilde k}\in S(M^{-1}, g_{\ep})$ 
such that $(\omega^{-1/2}\phi^{-n})\#r\#(\omega^{-1/2}\phi^{-n})=\{\ell, \psi\}\omega^{-1}\phi^{-2n}$  where $r-\{\ell, \psi\}\in S(M^{-1}, g_{\ep})$ is clear. Recalling Lemma \ref{lem:kokan} and applying Lemma \ref{lem:kihon:fu} we obtain
\begin{gather*}
|(\varPhi Au, [L, \varPhi]Lu)|\leq n(|\kappa|+CM^{-1})\|\varPhi^{\sharp}Au\|\|\varPhi^{\sharp}Lu\|
+CM\|\varPhi Au\|\|\varPhi Lu\|.
\end{gather*}
Since $|([\varPhi, L]Au, \varPhi Lu)|$ can be estimated in the same way we have
\begin{equation}
\label{eq:L:kokan:Phi}
\begin{split}
2|(\varPhi Au, [L, \varPhi]Lu)|+2|([\varPhi, L]Au, \varPhi Lu)|\\
\leq 2n(|\kappa|+CM^{-1})\big(\|\varPhi^{\sharp} Au\|^2
+\|\varPhi^{\sharp} Lu\|^2\big)\\
+CM\big(\|\varPhi Au\|^2+\|\varPhi Lu\|^2\big).
\end{split}
\end{equation}
Note that $[A, L]=-i\,\op{\dif_t\ell}$ and $\dif_t\ell\in S(\lr{\xi}, G)$. Write
\begin{equation}
\label{eq:dif:ltophi}
(1+k_1)\#(\omega^{1/2}\phi^n)\#\phi^{-n}\#\phi^{-n}\#(\dif_t\ell)\#\lr{\xi}^{-1}\#(\omega^{-1/2}\phi^n)
\#(1+k_2)=r
\end{equation}
such that $(\omega^{-1/2}\phi^{-n})\#r\#(\omega^{1/2}\phi^{-n})\#\lr{\xi}=\phi^{-n}\#\phi^{-n}\#(\dif_t\ell)$ where it is clear that $r-\dif_t\ell\lr{\xi}^{-1}\in 
 S(M^{-1}, g_{\ep})$. Noting \eqref{eq:ell:bar} and the constraint of the $t$ range \eqref{eq:t:seigen} we have 
\[
\big|\lr{\xi}^{-1}\dif_t\ell(t, x,\xi)\big|\leq c_0+CM^{-1},\quad c_0=|\dif_t\ell(0,e_d)|.
\]
Then it follows from Lemma \ref{lem:Ni:hon:2} that
\begin{equation}
\label{eq:A:kokan:L}
|(\varPhi[A, L]u, \varPhi Lu)|\leq (c_0+CM^{-1})\|\varPhi^{\flat}\lr{D} u\|\|\varPhi^{\sharp}Lu\|.
\end{equation}
From \eqref{eq:ene:id:K}, \eqref{eq:A:to:K}, \eqref{eq:L:kokan:Phi} and \eqref{eq:A:kokan:L} it follows that
\begin{lem}
\label{lem:ene:L}We have
\begin{gather*}
2{\mathsf{Im}}(\varPhi L^2u, \varPhi Au)\geq \dif_t\|\varPhi Lu\|^2+(2\theta-CM)\|\varPhi Lu\|^2\\
+2n(1-|\kappa|-c_0/2n-CM^{-1})\|\varPhi^{\sharp}Lu\|^2
-2n(|\kappa|+CM^{-1})\|\varPhi^{\sharp}Au\|^2\\-(c_0+CM^{-1})\| \varPhi^{\flat}\lr{D}u\|^2
-CM\|\varPhi Au\|^2.
\end{gather*}
\end{lem}
Note that from Corollary \ref{cor:lam:2} and Lemma \ref{lem:b:2} we have 
\begin{equation}
\label{eq:jd:b:seimitu}
\begin{split}
\dif_x^{\al}\dif_{\xi}^{\be}b\in S(\lr{\xi}^{1-|\be|}, \bg),\quad |\al+\be|=1,\\
\dif_{x_1}b\in S(M^{-1/2}\lr{\xi}, \bg), (a),\quad \dif_{\xi_j}b\in S(M^{-1/2}, \bg), \;\; j=1, d, (b).
\end{split}
\end{equation}
From Proposition \ref{pro:dif:omephi:matome} and Lemma \ref{lem:phi:cho:seimitu} it follows that
\begin{equation}
\label{eq:jd:phi:seimitu}
\begin{split}
\dif_x^{\al}\dif_{\xi}^{\be}(\omega^{-1/2}\phi^{-n})\in S(\omega^{-3/2}\lr{\xi}^{-|\be|}\phi^{-n}, g_{\ep}),\;\;|\al+\be|=1,\\
\dif_{\xi_j}(\omega^{-1/2}\phi^{-n})\in S(M^{-1}\omega^{-3/2}\lr{\xi}^{-1}\phi^{-n}, g_a),\quad j\neq 1,\;\; (a)\\
\dif_{x_j}(\omega^{-1/2}\phi^{-n})\in S(M^{-1}\omega^{-3/2}\phi^{-n}, g_b),\quad j\neq 1, d,\;\; (b).
\end{split}
\end{equation}
Since $g_{\ep}\leq \bg$ and $\omega$ and $\phi$ are $\bg$ admissible weights thanks to Proposition \ref{pro:omephi:to:g} one concludes from \eqref{eq:jd:b:seimitu} and \eqref{eq:jd:phi:seimitu} that
\begin{equation}
\label{eq:b:kokan:phi}
(\omega^{-1/2}\phi^{-n})\#b-b\#(\omega^{-1/2}\phi^{-n})\in S(M^{-1/2}\lr{\xi}\omega^{1/2}\phi^{-n}, \bg)
\end{equation}
where we have used $\omega^2\geq \lr{\xi}^{-1}$. Thus an application of Lemma \ref{lem:kihon:fu} shows 
\begin{equation}
\label{eq:Phi:kokan:B}
\|\varPhi^{\sharp}Bu\|\geq \|B\varPhi^{\sharp}u\|-CM^{-1/2}\|\varPhi^{\flat}\lr{D} u\|.
\end{equation}
Let ${\tilde B}=\op{{\tilde b}}$ where ${\tilde b}$ is given in Proposition \ref{pro:b:lam} such that $B\cdot{\tilde B}=1$ and ${\tilde B}\cdot B=1$. 
In view of \eqref{eq:b:sita} one sees ${\tilde b}^{-1}\in S(\lr{\xi}^{-1}\omega^{-1}, \bg)$ hence $
(\lr{\xi}\omega)\#{\tilde b}\in S(1, \bg)$. 
Therefore writing $\lr{\xi}\omega=(\lr{\xi}\omega)\#{\tilde b}\#b$ there is ${\hat c}>0$ such that
\begin{equation}
\label{eq:B:sita}
\|\op{\lr{\xi}\omega}u\|\leq \|Bu\|/{\hat c}.
\end{equation}
Writing $(\lr{\xi}\omega)\#(\omega^{-1/2}\phi^{-n})=(1+k)\#(\omega^{1/2}\phi^{-n})
\#\lr{\xi}$  it results
\begin{equation}
\label{eq:B:sita:D}
(1-CM^{-1})\|\varPhi^{\flat}\lr{D} u\|\leq \|\op{\lr{\xi}\omega}\varPhi^{\sharp}u\|.
\end{equation}
Replacing $u$ by $\varPhi^{\sharp}u$ in \eqref{eq:B:sita} we obtain from \eqref{eq:Phi:kokan:B} and \eqref{eq:B:sita:D} that
\begin{lem}
\label{lem:B:sita}There are ${\hat c}>0$, $C>0$ such that
\begin{equation}
\label{eq:B:wo:D}
{\hat c}(1-CM^{-1/2})\|\varPhi^{\flat}\lr{D} u\|\leq \|\varPhi^{\sharp} Bu\|.
\end{equation}
\end{lem}
Denoting $\varPhi^{\flat\flat}=\op{\omega\phi^{-n}}$ the same argument  shows that
\begin{equation}
\label{eq:B:wo:D:bis}
{\hat c}(1-CM^{-1/2})\|\varPhi^{\flat\flat}\lr{D} u\|\leq \|\varPhi  B u\|.
\end{equation}
It is clear that $b\#\phi^{-n}-\phi^{-n}\#b\in S(M^{-1/2}\phi^{-n}\omega^{-1}, \bg)$ from the same argument proving \eqref{eq:b:kokan:phi}. Write
\[
(1+k)\#(\omega^{1/2}\phi^n)\#\phi^{-n}\#(b\#\phi^{-n}-\phi^{-n}\#b)\#(\omega^{1/2}\phi^n)\#(1+{\tilde k})=r
\]
such that $(\omega^{-1/2}\phi^{-n})\#r\#(\omega^{-1/2}\phi^{-n})=\phi^{-n}\#(b\#\phi^{-n}-\phi^{-n}\#b)$ where $r\in S(M^{-1/2}, \bg)$. 
Therefore one has
\begin{gather*}
|(\varPhi Au, [B, \varPhi]Bu)|
\leq \|\varPhi^{\sharp} Au\|\|\op{r}\varPhi^{\sharp}Bu\|\\
\leq CM^{-1/2}\big(\|\varPhi^{\sharp} Au\|^2+\|\varPhi^{\sharp} Bu\|^2\big).
\end{gather*}
Repeating the same arguments again we have
\begin{equation}
\label{eq:B:kokan:Phi}
\begin{split}
|(\varPhi Au, [B, \varPhi]Bu)|+|([\varPhi, B]Au, \varPhi B u)|\\
\leq CM^{-1/2}\big(\|\varPhi^{\sharp} Au\|^2+\|\varPhi^{\sharp} B u\|^2\big).
\end{split}
\end{equation}
Write $
(1+k)\#(\omega^{1/2}\phi^n)\#\phi^{-n}\#\phi^{-n}\#(\dif_tb)\#\lr{\xi}^{-1}\#(\omega^{-1/2}\phi^n)\#(1+{\tilde k})=r$ 
such that $(\omega^{-1/2}\phi^{-n})\#r\#(\omega^{1/2}\phi^{-n})\#\lr{\xi}=\phi^{-n}\#\phi^{-n}\#(\dif_tb)$. Here we note
\begin{lem}
\label{lem:b:tbibun}Notations being as above we have $r-\lr{\xi}^{-1}\dif_tb\in S(M^{-1/2}, \bg)$.
\end{lem}
\begin{proof}Write $(1+k)\#(\omega^{1/2}\phi^n)\#\phi^{-n}\#\phi^{-n}=\omega^{1/2}\phi^{-n}+l$ with $l\in S(M^{-1}\omega^{1/2}\phi^{-n}, g_{\ep})$ and $\lr{\xi}^{-1}\#(\omega^{-1/2}\phi^n)\#(1+{\tilde k})=\lr{\xi}^{-1}\omega^{-1/2}\phi^{n}+{\tilde l}$ with ${\tilde l}\in S(M^{-1}\lr{\xi}^{-1}\omega^{-1/2}\phi^{n}, g_{\ep})$ such that  $r=(\omega^{1/2}\phi^{-n}+l)\#(\dif_tb)\#(\lr{\xi}^{-1}\omega^{-1/2}\phi^n+{\tilde l})$. Thanks to Corollary \ref{cor:b:tbibun} and \eqref{eq:jd:phi:seimitu} it follows that
\[
(\omega^{1/2}\phi^{-n})\#(\dif_tb)-(\dif_tb)\#(\omega^{1/2}\phi^{-n})\in S(M^{-1/2}\lr{\xi}\omega^{1/2}\phi^{-n}, \bg)
\]
hence we have
\begin{equation}
\label{eq:db:atukai}
\begin{split}
r=(\dif_tb)\#(\omega^{1/2}\phi^{-n})\#(\lr{\xi}^{-1}\omega^{-1/2}\phi^n)+R
=(\dif_tb)\#\lr{\xi}^{-1}+{\tilde R}
\end{split}
\end{equation}
where ${\tilde R}\in S(M^{-1/2}, \bg)$. Since $(\dif_tb)\#\lr{\xi}^{-1}-\lr{\xi}^{-1}\dif_tb\in S(M^{-1/2}, \bg)$ the proof is completed.
\end{proof}
Since $\lr{\xi}^{-1}\dif_tb\in S(1, \bg)$ in view of Corollary \ref{cor:b:tbibun} from the $L^2$ boundedness theorem there are $c>0$ and $l\in\N$ such that 
\begin{equation}
\label{eq:c:pm:teigi}
\|\op{\lr{\xi}^{-1}\dif_tb}u\|\leq c\big|\lr{\xi}^{-1}\dif_tb\big|^{(l)}_{S(1, \bg)}\|u\|=c_1\|u\|.
\end{equation}
Then from \eqref{eq:db:atukai} it follows that
\begin{equation}
\label{eq:A:kokan:B}
|(\varPhi_n[A, B]u, \varPhi Bu)|\leq (c_1+CM^{-1/2})\|\varPhi^{\sharp}Bu\|\|\varPhi^{\flat}\lr{D} u\|.
\end{equation}
From \eqref{eq:ene:id:K}, \eqref{eq:A:to:K}, \eqref{eq:B:wo:D}, \eqref{eq:B:kokan:Phi} and \eqref{eq:A:kokan:B} we have
\begin{lem}
\label{lem:ene:B}We have
\begin{gather*}
2{\mathsf{Im}}(\varPhi B^2u, \varPhi Au)\geq \dif_t\|\varPhi Bu\|^2+2\theta\|\varPhi B_u\|^2\\
+n({\hat c}-c_1/n-CM^{-1/2})\|\varPhi^{\flat}\lr{D}u\|^2\\
+n(1-c_1/n-CM^{-1/2})\|\varPhi^{\sharp}Bu\|^2
-CM^{-1/2}\|\varPhi^{\sharp}Au\|^2.
\end{gather*}
\end{lem}
%

Since
\begin{gather*}
-2{\mathsf{Im}}(\varPhi Au, \varPhi u)=\dif_t\|\varPhi u\|^2+2\theta \|\varPhi u\|^2
+2{\mathsf{Im}}([A, \varPhi]u, \varPhi u)
\end{gather*}
replacing $u$ by $Au$ 
it follows from \eqref{eq:A:to:K} that
\begin{equation}
\label{eq:A:sita}
\begin{split}
-2{\mathsf{Im}}(\varPhi A^2u, \varPhi Au)
\geq \dif_t\|\varPhi Au\|^2+2\theta \|\varPhi A u\|^2\\
+2n(1-CM^{-1/2})\|\varPhi^{\sharp}A u\|^2.
\end{split}
\end{equation}
Then from Lemmas \ref{lem:ene:L} and \ref{lem:ene:B} we conclude
\begin{prop}
\label{pro:ene:ALB}
We have
\begin{gather*}
2{\mathsf{Im}}(\varPhi (-A^2+L^2+B^2)u, \varPhi Au)\geq \dif_t\big(\|\varPhi Lu\|^2+\|\varPhi B u\|^2+\|\varPhi Au\|^2\big)\\
+(2\theta-CM)\big(\|\varPhi L u\|^2+\|\varPhi B u\|^2+\|\varPhi Au\|^2\big)\\
+2n(1-|\kappa|-c_0/2n-CM^{-1/2})\|\varPhi^{\sharp}L u\|^2\\
+2n(1-|\kappa|-CM^{-1/2})\|\varPhi^{\sharp} Au\|^2\\
+n({\hat c}-c_0/n-c_1/n-CM^{-1/2})\|\varPhi^{\flat}\lr{D}u\|^2\\
+n(1-c_1/n-CM^{-1/2})\|\varPhi^{\sharp}B u\|^2.
\end{gather*}
\end{prop}
%

Since $-2(\varPhi A u, \varPhi u)
\geq \dif_t\|\varPhi u\|^2+2\theta \|\varPhi u\|^2
$ if $CM^{-1/2}\leq 1$ then 
\begin{equation}
\label{eq:theta:nijyo}
\|\varPhi Au\|^2\geq \theta \dif_t\|\varPhi u\|^2+\theta^2\|\varPhi u\|^2.
\end{equation}
Consider the lower order term $B_0D_t+B_1=B_0A+B_1+i\theta B_0$. Write
\begin{gather*}
(1+k)\#(\omega^{1/2}\phi^n)\#\phi^{-n}\#\phi^{-n}\#{\tilde b_1}\#\lr{\xi}^{-1}\#(\omega^{-1/2}\phi^n)\#(1+{\tilde k})=r
\end{gather*}
with $r\in S(1, \bg)$ such that $(\omega^{-1/2}\phi^{-n})\#r\#(\omega^{1/2}\phi^{-n})\#\lr{\xi}=\phi^{-n}\#\phi^{-n}\#{\tilde b}_1$. We make a closer look at $r$. 
\begin{lem}
\label{lem:teikai:seimitu}Notations being as above we have $r-\lr{\xi}^{-1}{\tilde b_1}\in S(M^{-1/2}, \bg)$.
\end{lem}
\begin{proof}First note that ${\tilde b_1}=d_1-{\tilde r}$ with some $d_1\in S(\lr{\xi}, g_{\ep})$ and ${\tilde r}$ given in \eqref{eq:q:to:B}. Since ${\tilde r}=b\#b-b^2$ thanks to Corollary \ref{cor:lam:2} it follows that 
\begin{gather*}
\dif_x^{\al}\dif_{\xi}^{\be}\dif_{x_1}{\bar b}\in S(M^{-1/2}\lr{\xi}^{(|\al|-|\be|)/2}, \bg),\quad (a),\\
\dif_x^{\al}\dif_{\xi}^{\be}\dif_{\xi_j}{\bar b}\in S(M^{-1/2}\lr{\xi}^{-1}\lr{\xi}^{(|\al|-|\be|)/2}, \bg), \quad j=1, d,\;\;(b)
\end{gather*}
which together with Lemma \ref{lem:b:2} proves that $\dif_{x_1}{\tilde r}\in S(M^{-1/2}\lr{\xi}^{3/2}, \bg)$, $(a)$ and $\dif_{\xi_j}{\tilde r}\in S(M^{-1/2}\lr{\xi}^{1/2}, \bg)$ for $j=1, d$, $(b)$. Applying the same arguments proving Lemma \ref{lem:b:tbibun} we conclude the assertion.
\end{proof}
Since $\op{\lr{\xi}^{-1}{\tilde b_1}}$ is $L^2$ bounded, denoting the bound by ${\bar c}$,  we have
\begin{equation}
\label{eq:def:bar:c}
\|\op{\lr{\xi}^{-1}{\tilde b_1}}u\|\leq {\bar c}\,\|u\|
\end{equation}
hence
\begin{equation}
\label{eq:B:1:A}
2|(\varPhi B_1 u, \varPhi Au)|
\leq ({\bar c}+CM^{-1/2})(\|\varPhi^{\sharp} Au\|^2+\|\varPhi^{\flat}\lr{D} u\|^2).
\end{equation}
Writing $\phi^{-n}\#{r_0}\#\phi^{-n}=\phi^{-n}\#\phi^{-n}\#{\tilde b_0}$ with $r_0\in S(1, \bg)$ it results
\begin{equation}
\label{eq:B:0:A}
2|(\varPhi B_0 Au, \varPhi Au)|
\leq CM\|\varPhi Au\|^2.
\end{equation}
Similarly one has
\begin{equation}
\label{eq:B:the:A}
2\theta |(\varPhi B_0 u, \varPhi Au)|
\leq CM(\theta^{3/2}\|\varPhi u\|^2+\theta^{1/2}\|\varPhi Au\|^2).
\end{equation}
It is also easy to see that
\begin{gather*}
2|(\varPhi (-A^2+L+B^2)u, \varPhi Au)|\leq M^{1/2}\|\varPhi^{\flat}(-A^2+L+B^2)u\|^2\\
+M^{-1/2}(1+CM^{-1/2})\|\varPhi^{\sharp}Au\|^2.
\end{gather*}
Therefore from Proposition \ref{pro:ene:ALB} and \eqref{eq:theta:nijyo} we arrive at 
\begin{gather*}
M^{1/2}\|\varPhi^{\flat}{\hat P_{\theta}}u\|^2\geq \dif_t\big(\|\varPhi L u\|^2+\|\varPhi B u\|^2+\|\varPhi Au\|^2+\theta \|\varPhi u\|^2\big)\\
+\theta(1-CM^2\theta^{-1}-CM\theta^{-1/2})\big(\|\varPhi L u\|^2+\|\varPhi B u\|^2+\|\varPhi Au\|^2\big)\\
+\theta^2(1-CM\theta^{-1/2})\|\varPhi u\|^2\\
+2n(1-|\kappa|-c_0/2n-CM^{-1/2})\|\varPhi^{\sharp}L u\|^2\\
+2n(1-|\kappa|-{\bar c}/2n-CM^{-1/2})\|\varPhi^{\sharp}Au\|^2\\
+n({\hat c}-c_0/n-c_1/n-{\bar c}/n-CM^{-1/2})\|\varPhi^{\flat}\lr{D}u\|^2\\
+n(1-c_1/n-CM^{-1/2})\|\varPhi^{\sharp}B u\|^2.
\end{gather*}
Here writing $(\omega^{1/2}\phi^{-n})\#\lr{\xi}=(1+k)\#(\omega^{1/2}\lr{\xi}^{1/4})\#\phi^{-n}\#\lr{\xi}^{3/4}$ and noting $\omega^{1/2}\lr{\xi}^{1/4}\geq 1$ one has by Lemma \ref{lem:Ni:hon:2} 
\[
\|\varPhi^{\flat}\lr{D} u\|\geq (1-CM^{-1})\|\varPhi\lr{D}^{3/4}u\|.
\]
Similarly we see $\|\varPhi^{\flat\flat}\lr{D} u\|\geq (1-CM^{-1})\|\varPhi\lr{D}^{1/2}u\|$.
Thus we first choose $n$ such that
\begin{gather*}
1-|\kappa|-c_0/2n>0,\quad 1-|\kappa|-{\bar c}/2n>0,\\
{\hat c}-c_0/n-c_1/n-{\bar c}/n>0,\quad 1-c_1/n>0
\end{gather*}
and {\it fix }such a $n$.  Next we choose $M$ such that the above inequalities remain to be positive after subtracting $CM^{-1/2}$ from each inequality 
and {\it fix} such a $M$ then choose $\gamma$ such that $\gamma \geq M^{4}$ and $\gamma \geq {\bar\lam}M^{2}$ and {\it fix} $\gamma$, still $\theta$ is  assumed   to be free. Once $M$ and $\gamma$ are fixed we have 
\[
g_0/C\leq G\leq Cg_0,\quad \olr{\xi}^s/C_s\leq \lr{\xi}^s\leq C_s\olr{\xi}^s
\]
where $g_0=|dx|^2+\olr{\xi}^{-2}|d\xi|^2$. Now summarize what we have proved
\begin{prop}
\label{pro:bibun:ene:matome}There exist $C>0$, $c>0$ and $\theta_0>0$ such that
\begin{gather*}
C\|\varPhi^{\flat}{\hat P_{\theta}}u\|^2\geq \dif_t\big(\|\varPhi L u\|^2+\|\varPhi B u\|^2+\|\varPhi Au\|^2+\theta \|\varPhi u\|^2\big)\\
+c\,\theta\big(\|\varPhi L u\|^2+\|\varPhi B u\|^2+\|\varPhi Au\|^2+\|\varPhi \olr{D}^{1/2}u\|^2+\theta \|\varPhi u\|^2\big)\\
+c\,\big(\|\varPhi^{\sharp}L u\|^2+\|\varPhi^{\sharp}Au\|^2+\|\varPhi^{\sharp}B u\|^2+\|\varPhi\olr{D}^{3/4}u\|^2\big)
\end{gather*}
for $\theta\geq \theta_0$.
\end{prop}
Next we estimate $\olr{D}^su$. Since $\olr{D}^s{\hat P_{\theta}}={\hat P_{\theta}}\olr{D}^s+[\olr{D}^s, {\hat P}]$ we study $[\olr{D}^s, L^2]$. Since $\ell\#\ell-\ell^2\in S(1, g_{0})$ is clear then
\[
\olr{\xi}^s\#\ell\#\ell-\ell\#\ell\#\olr{\xi}^s-(\olr{\xi}^2\#\ell^2-\ell^2\#\olr{\xi}^s)\in S(\olr{\xi}^s, g_{0}).
\]
It is also easy to see that $
\olr{\xi}^s\#\ell^2-\ell^2\#\olr{\xi}^s=a\ell+r$ where $a\in S(\olr{\xi}^s, g_0)$ and $r\in S(\olr{\xi}^s, g_{0})$. Since one can write $
a\ell=(a\olr{\xi}^{-s})\#\ell\#\olr{\xi}^s+{\tilde r}$ 
with ${\tilde r}\in S(\olr{\xi}^s, g_{0})$ we conclude that
\begin{equation}
\label{eq:l:D:s}
\begin{split}
\big|(\varPhi[\olr{D}^s, L^2]u, \varPhi A\olr{D}^s u)\big|\leq C\big(\|\varPhi A\olr{D}^su\|^2\\
+\|\varPhi L \olr{D}^su\|^2+\|\varPhi \olr{D}^su\|^2\big).
\end{split}
\end{equation}
%
%
 \begin{lem}
 \label{lem:b:D:s}We have
 \begin{gather*}
 \big|(\varPhi [\olr{D}^s, B^2]u, \varPhi A\olr{D}^s u)\big|\leq C\|\varPhi A\olr{D}^su\|
\|\varPhi B\olr{D}^su\|.
 \end{gather*}
 \end{lem}
 \begin{proof}Note that $[\olr{D}^s, B^2]=[\olr{D}^s, B]B+B[\olr{D}^s, B]$. 
 From Lemma \ref{lem:b:2} we see $\olr{\xi}^s\#b-b\#\olr{\xi}^s\in S(\olr{\xi}^s, \bg)$. Thanks to Proposition \ref{pro:lam:1} one has $r=(\olr{\xi}^s\#b-b\#\olr{\xi}^s)\#b\in S(b\olr{\xi}^s, \bg)$. Applying Proposition \ref{pro:b:lam} one can write
 \[
(\olr{\xi}^s\#b-b\#\olr{\xi}^s)\#b=r\#\olr{\xi}^{-s}\#{\tilde b}\#b\#\olr{\xi}^s
 \]
 where $r\#\olr{\xi}^{-s}\#{\tilde b}\in S(1, \bg)$. Then writing $\phi^{-n}\#(r\#\olr{\xi}^{-s}\#{\tilde b})={\tilde r}\#\phi^{-n}$ with ${\tilde r}\in S(1, \bg)$ we conclude
 \[
\big|(\varPhi[\olr{D}^s, B]B u, \varPhi A\olr{D}^s)\big|\leq C\|\varPhi B\olr{D}^su\|\|\varPhi A\olr{D}^su\|.
\]
Repeating the same arguments to $B[\olr{D}^s, B]$ we end the proof. 
 \end{proof}
 For commutators coming from lower order term it is easy to see
 \begin{equation}
 \label{eq:kokan:teikai:s}
 \begin{split}
 \big|(\varPhi[\olr{D}^s, B_0]Au, \varPhi A\olr{D}^su)\big|\leq C\|\varPhi  A\olr{D}^su\|^2,\\
\big|(\varPhi [\olr{D}^s, B_0]u, \varPhi A\olr{D}^s u)\big|\leq C\|\varPhi \olr{D}^su\|\|\varPhi A\lr{D}^s u\|,\\
\big|(\varPhi [\olr{D}^s, B_1]u, \varPhi A\olr{D}^s u)\big|\leq C\|\varPhi \olr{D}^{s+1/2}u\|\|\varPhi A\olr{D}^s u\|.
 \end{split}
 \end{equation}
It follows from \eqref{eq:l:D:s}, \eqref{eq:kokan:teikai:s} and Lemma \ref{lem:b:D:s} that $|(\varPhi [\olr{D}^s, {\hat P}]u, \varPhi A\olr{D}^s u)|$ is controlled by the second term on the right-hand side of Proposition \ref{pro:bibun:ene:matome} with $\olr{D}^su$ in place of $u$, choosing $\theta$ suitably large. 

Recalling $Ae^{-\theta t}=e^{-\theta t}D_t$ one has from Proposition \ref{pro:bibun:ene:matome} that
\begin{equation}
\label{eq:ene:simple}
\begin{split}
Ce^{-2\theta t}\|\varPhi^{\flat}\olr{D}^s{\hat P}u\|^2\geq \dif_t e^{-2\theta t}\big(\|\varPhi L \olr{D}^su\|^2+\|\varPhi B \olr{D}^su\|^2\\
+\|\varPhi\olr{D}^sD_tu\|^2+\theta \|\varPhi\olr{D}^su\|^2\big).
\end{split}
\end{equation}
Here we note
\begin{lem}
\label{lem:Phi:ue:sita}Let $n\geq 1$. There is $C>0$ such that $
C\|\varPhi B v\|\geq \|\olr{D}v\|$.
\end{lem}
\begin{proof}Since $\phi\leq 2\omega$ and $\omega\phi^{-n}\geq 2^{-n}\omega^{-n+1}\geq (2C)^{-n+1}/2$. Thus the proof follows from \eqref{eq:B:wo:D:bis} and Lemma \ref{lem:m:1:2}.
\end{proof}
Similarly from \eqref{eq:phi:uesita}, using Lemma \ref{lem:m:1:2}, we have
\begin{equation}
\label{eq:phi:norm:uesita}
\|v\|/C\leq \|\varPhi v\|,\;\|\varPhi^{\flat}v\|\leq C\|\olr{D}^nv\|, 
\;\;n\geq 1/2.
\end{equation}
\begin{definition}
\label{dfn:calH}\rm We denote $\|u\|_s=\|\olr{D}^su\|$ and by $H^s=H^s(\R^d)$ the $L^2$ based Sobolev space of order $s$. Denote by ${\mathcal H}_{-k,s}(\delta_1,\delta_2)$ the set of all $f$ such that
\[
(t-\delta_1)^{-k}\olr{D}^{s}f\in L^2((\delta_1,\delta_2)\times \R^d).
\]
\end{definition}
Assume  $D_t^ju\in {\mathcal H}_{-k, s+2-j}(\delta_1, \delta_2)$, $j=0, 1, 2$. From this one sees that $\lim_{t\to +\delta_1}\|D_t^ju(t)\|_{s+1-j}$, $j=0, 1$ exists which is $0$ for $k>0$. Using this we see $
\lim_{t\to +\delta_1}(t-\delta_1)^{-k}\|D_t^ju(t)\|_{s+1-j}=0$, $j=0, 1$. 
Let ${\hat\tau}$ be any point with $|{\hat\tau}|<\delta$. Multiply \eqref{eq:ene:simple}  by $(t-{\hat\tau})^{-2k}$ and integrate in $t$ from ${\hat \tau}$ to $t$ we obtain
\begin{prop}
\label{pro:ene:sekibun}For any $s\in\R$ there is $C$ such that
\begin{equation}
\label{eq:ene:sekibun}
\begin{split}
(t-{\hat\tau})^{-2k}\big(\|D_tu(t)\|_s^2+\|u(t)\|_{s+1}^2\big)\\
+\int_{\hat\tau}^t(\tau-{\hat\tau})^{-2k-1}\big(\|D_tu(\tau)\|_s^2+\|u(\tau)\|_{s+1}^2\big)d\tau\\
\leq C\int_{\hat\tau}^t(\tau-{\hat\tau})^{-2k}\|{\hat P}u(\tau)\|_{n+s}^2d\tau
\end{split}
\end{equation}
for any $u$ with $D_t^ju\in {\mathcal H}_{-k, n+s+2-j}({\hat\tau}, \delta)$, $j=0,1,2$.
\end{prop}
Consider the adjoint ${\hat P^*}$ of ${\hat P}$. Denoting ${\check\varPhi}=\op{\phi^n}$, ${\check \varPhi^{\flat}}=\op{\omega^{1/2}\phi^n}$ and ${\check \varPhi^{\sharp}}=\op{\omega^{-1/2}\phi^n}$ a repetition of the same argument gives  
\begin{equation}
\label{eq:ene:simple:ad}
\begin{split}
Ce^{2\theta t}\|{\check\varPhi^{\flat}}\olr{D}^s{\hat P^*}u\|^2\geq -\dif_t e^{2\theta t}\big(\|{\check\varPhi}L\olr{D}^su\|^2+\|{\check\varPhi}B \olr{D}^su\|^2\\
+\|{\check\varPhi}\olr{D}^sD_tu\|^2+\theta \|{\check\varPhi}\olr{D}^su\|^2\big).
\end{split}
\end{equation}
Since $2\phi\omega\geq \lr{\xi}^{-1}$ repeating  similar arguments one has
\begin{equation}
\label{eq:Phi:uesita:ad}
\|\olr{D}^{-n}v\|/C\leq \|{\check\varPhi }v\|\leq C\|v\|,\quad \|\olr{D}^{-n+1}v\|\leq C\|{\check\varPhi}B v\|,\;n\geq 1.
\end{equation}
Multiply \eqref{eq:ene:simple:ad} by $(t-{\hat\tau})^{2k+1}$ and integrate  in $I=({\hat\tau}, \delta)$ we have
\begin{prop}
\label{pro:ene:sekibun:ad}For any $s\in\R$ there is $C$ such that
\begin{gather*}
\int_I(\tau-{\hat\tau})^{2k}\big(\|D_tu(t)\|_{-n+s}^2+\|u(t)\|_{-n+s+1}^2\big)dt\\
\leq C\int_I(\tau-{\hat\tau})^{2k+1}\|{\hat P^*}u(t)\|_s^2dt,\quad u\in C_0^{\infty}(I\times\R^d).
\end{gather*}
\end{prop}
%
\subsection{Local existence theorem}

From Proposition \ref{pro:ene:sekibun:ad} we have
\begin{align*}
\big|\int_{I}(f,v)dt\big|\leq \big(\int_{I}(t-{\hat\tau})^{-2k}\|f\|_{n+k+s+1}^2dt\big)^{1/2}\big(\int_{I}(t-{\hat\tau})^{2k}\|v\|_{-n-k-s-1}^2dt\big)^{1/2}\\
\leq C\big(\int_{I}(t-{\hat\tau})^{-2k}\|f\|_{n+k+s+1}^2dt\big)^{1/2}\big(\int_{I}(t-{\hat\tau})^{2k+1}\|{\hat P^*}v\|_{-2n-k-s-2}^2dt\big)^{1/2}
\end{align*}
for any $v\in C_0^{\infty}(I\times \R^d)$ and $f\in {\mathcal H}_{-k, n+k+s+1}(I)$. Using the Hahn-Banach theorem to extend the anti-linear form in ${\hat P^*}v$;
\begin{equation}
\label{eq:keisiki}
{\hat P^*}v\mapsto \int_{I}(f, v)dt
\end{equation}
we conclude that there is some $u\in {\mathcal H}_{-k -1/2, 2n+k+s+2 }(I)$  such that
\[
\int_{I}(f, v)dt=\int_{I}(u,{\hat P^*}v)dt,\quad v\in C_0^{\infty}(I\times \R^d).
\]
This implies that ${\hat P}u=f$. Since $u\in {\mathcal H}_{0, 2n+k+s+2}(I)$ and $f\in {\mathcal H}_{0, n+k+s+1}(I)$  it follows from \cite[Theorem B.2.9]{Hobook} that $D_t^ju\in {\mathcal H}_{0, n+k+s+3-j}(I)$ for $ j=0,1,2,\ldots$. Thus with $w_j=\olr{D}^{n+s+2-j}D_t^ju$ one has $D_t^iw_j\in L^2(I\times \R^d)$ for $i=0,\ldots, k+1$ hence $D_t^iw_j({\hat\tau})$ exists in $L^2(\R^d)$ which is $0$ for $i=0,\ldots, k$ since $w_j\in {\mathcal H}_{-k-1/2, 0}(I)$. Thus one can write $w_j(t)=\int_{\hat\tau}^t(t-\tau)^{k}\dif_t^{k+1}w_j(\tau)d\tau/k!$. Thus one concludes that $D_t^ju\in {\mathcal H}_{-k-1/2, n+s+2-j}(I)$  for $0\leq j\leq 2$  then \eqref{eq:ene:sekibun} holds for this $u$. Now let $f\in {\mathcal H}_{-k, n+s}(I)$. Take a rapidly decreasing function $\rho(\xi)$ with $\rho(0)=1$ then $f_{\ep}=\rho(\ep D)f\in {\mathcal H}_{-k, 2n+k+s+2}(I)$ and $f_{\ep}\to f$ in ${\mathcal H}_{-k, n+s}(I)$. As just proved above there is $u_{\ep}$ satisfying ${\hat P}u_{\ep}=f_{\ep}$ and \eqref{eq:ene:sekibun}. Therefore choosing a weakly convergent subsequence $\{u_{\ep'}\}$ one can conclude
\begin{them}
\label{thm:pre:sonzai:s} For any $s\in \R$ and  any $f\in {\mathcal H}_{-k, n+s}(I)$ there exists a unique $u$ with $D_t^ju\in {\mathcal H}_{-k-1/2, s+1-j}(I)$, $j=0,1,2$, satisfying ${\hat P}u=f$ and \eqref{eq:ene:sekibun}.
\end{them}
Thanks to Theorem \ref{thm:pre:sonzai:s} one can define the solution map
\[
{\hat G}({\hat\tau}): {\mathcal H}_{-k,n+s}(I)\ni f\mapsto u\in {\mathcal H}_{-k-1/2, s+1}(I),\quad I=({\hat\tau}, \delta).
\]
We shall keep ${\hat\tau}$ fixed in the following discussion and therefore we write  ${\hat G}$ dropping ${\hat \tau}$. This solution operator ${\hat G}$ verifies  
\begin{equation}
\label{eq:sol:op}
\sum_{j=0}^1\int_{\hat\tau}^t(\tau-{\hat\tau})^{-2k-1}\|D_t^j{\hat G}f(\tau)\|^2_{s+1-j}\leq C\int_{\hat\tau}^t(\tau-{\hat\tau})^{-2k}\|f(\tau)\|^2_{n+s}
\end{equation}
and has (microlocal) finite propagation speed. We state this property without proof (for a proof see \cite{Ni:book}).
\begin{prop}
\label{pro:yugen:denpa}Notations being as above and let $\Gamma_i$ $(i=1, 2, 3 )$ be open conic sets in $\R^d\times (\R^d\setminus\{0\})$ with relatively compact basis such that $\Gamma_1\Subset \Gamma_2\Subset \Gamma_3$ and  $h_i(x,\xi)\in S(1, g_0)=S^0$ with ${\rm supp}\,h_1\subset \Gamma_1$, ${\rm supp}\,h_2\subset \Gamma_3\setminus \Gamma_2$. Then there exists $\delta'=\delta'(\Gamma_i)>0$ such that for any $r$, $s$ one can find $C>0$ such that
\begin{gather*}
\sum_{j=0}^1\int_{\hat\tau}^t(\tau-{\hat\tau})^{-2k-1}\|\op{h_2}D_t^j{\hat G}\,\op{h_1}f(\tau)\|_{r-j}^2d\tau\\
\leq C\int_{\hat\tau}^t(\tau-{\hat\tau})^{-2k}\|f(\tau)\|_{s}^2d\tau,\quad {\hat\tau}< t\leq {\hat\tau}+\delta',\;\;f\in {\mathcal H}_{-k, s}({\hat\tau},  {\hat\tau}+\delta').
\end{gather*}
\end{prop}
Recall $(Tu)(t, x)=u(t, \kappa(x))$. Let $R_{\bar\xi}=P-T{\hat P}T^{-1}$ then with $G_{\bar\xi}=T{\hat G}T^{-1}$ we have
\[
PG_{\bar\xi}=I+R_{\bar\xi}G_{\bar\xi}
\]
where it is clear that $G_{\bar\xi}$ verifies \eqref{eq:sol:op}. Since $R_{\bar\xi}=T(\sum_{j=1}^2a_j(t, x, D)D^{2-j}_t)T^{-1}$ with $a_j\in S^j\cap S^{-\infty}(W_{M, \gamma})$ applying the description of the wave front set of $Tu$ (e.g.\cite[Theorem 8.2.4]{Hobook:1}) one can find a conic neighborhood $W_{\bar\xi}$ of $(0, {\bar \xi})$ such that for any $h(x, \xi)\in S^0$ supported in $W_{\bar\xi}$ we have
\begin{equation}
\label{eq:WF:henkan}
\| R_{\bar\xi}\,\op{h}u\|_p\precsim (\|D_tu\|_{q-1}+\|u\|_q),\quad \forall p, q\in\R.
\end{equation}
 It is not difficult to prove that $G_{{\bar\xi}}$ has (microlocal) finite propagation speed.

\begin{them}
\label{thm:sonzai:ippan} Assume that every singular point of $p(0, 0, \tau, \xi)=0$  is effectively hyperbolic. Then there exist $\delta>0$, $n>0$ and a neighborhood $U$ of $x=0$ such that for every $f\in {\mathcal H}_{-k,s}({\hat\tau},\delta)$ with $|{\hat\tau}|<\delta$ there exists $u$ with $D_t^j u\in {\mathcal H}_{-k, -n+s+1-j}({\hat\tau},\delta)$, $j=0,1$, satisfying $
Pu=f$ in $({\hat\tau},\delta)\times U$.
\end{them}
\begin{proof}Recall that we have proved that for any $|\eta|=1$ one can find a conic neighborhood  $W_{\eta}$ of $(0, \eta)$, a positive constant $\delta_{\eta}>0$ and 
a solution operator $G_{\eta}({\hat\tau})$ with (microlocal) finite propagation speed satisfying \eqref{eq:sol:op} such that
\[
PG_{\eta}=I+R_{\eta}G_{\eta},\quad |t|\leq \delta_{\eta}
\]
where $R_{\eta}$ satisfies \eqref{eq:WF:henkan} for $h\in S^0$ with ${\rm supp }\,h\subset W_{\eta}$. 
We can choose a finite number of $\eta_i$ such that $\cup_iW_{\eta_i}\supset U\times (\R^d\setminus \{0\})$, where $U$ is a neighborhood of 
$x=0$. Now take another open conic covering $\{V_i\}$ of $U\times (\R^d\setminus\{0\})$ with $V_i\Subset W_{\eta_i}$, and a partition of unity $\{\alpha_i(x,\xi)\in S^0\}$ subordinate to $\{V_i\}$ so that $
\sum_i\alpha_i(x,\xi)=\alpha(x)$ 
where $\alpha(x)$ is equal to $1$ in a neighborhood of $x=0$. Denoting
\[
G=\sum_iG_{\eta_i}\op{\alpha_i}
\]
we have $P G f=\sum_iPG_{\eta_i}\op{\alpha_i}f
=\alpha(x)f-Rf$
with $R=-\sum_i R_{\eta_i}G_{\eta_i}\op{\alpha_i}$. Now choosing $\chi_i\in S^0$ supported in $W_{\eta_i}$ such that $V_i\Subset\{\chi_i=1\}$ and writing $R_{\eta_i}G_{\eta_i}\op{\alpha_i}=R_{\eta_i}(\op{\chi_i}+\op{1-\chi_i})G_{\eta_i}\op{\alpha_i}$ it follows from (microlocal) finite propagation speed and \eqref{eq:WF:henkan} that there exists $\delta'>0$ such that
\[
\int_{\hat\tau}^t(\tau-{\hat\tau})^{-2k-1}\|Rf(\tau)\|_{s}^2d\tau\leq C\int_{\hat\tau}^t(\tau-{\hat\tau})^{-2k}\|f(\tau)\|_{s}^2d\tau
\]
for ${\hat\tau}\leq t\leq {\hat\tau}+\delta'$. Choosing $0<\delta_1\leq \delta'$ such that $\delta_1C\leq 1/2$ one has 
\[
\int_{\hat\tau}^t(\tau-{\hat\tau})^{-2k}\|Rf(\tau)\|_{s}^2d\tau\leq \frac{1}{2}\int_{\hat\tau}^t(\tau-{\hat\tau})^{-2k}\|f(\tau)\|_{s}^2d\tau
\]%
for $f\in {\mathcal H}_{-k, s}({\hat\tau}, {\hat\tau}+\delta_1)$. With $
S=\sum_{k=0}^{\infty}R^k$ 
we have $Sf\in {\mathcal H}_{-k, s}({\hat\tau}, {\hat\tau}+\delta_1)$ and
\[
\int_{\hat\tau}^t(\tau-{\hat\tau})^{-2k}\|Sf(\tau)\|_{s}
\leq 2\int_{\hat\tau}^t(\tau-{\hat\tau})^{-2k}\|f(\tau)\|_{s}.
\]
 Let $\gamma(x)\in C_0^{\infty}(\R^d)$ be equal to $1$ near $x=0$ such that ${\rm supp}\,\gamma\Subset \{\alpha=1\}$. Since $\gamma(\alpha-R)S=\gamma(I-R)S=\gamma$ we have $
\gamma(x)PGSf=\gamma(x)f$, that is $
P\big(GSf\big)=f$ on $\{\gamma(x)=1\}$.
With $u=G S f$ one has
\begin{gather*}
\sum_{j=0}^{1}\int_{\hat\tau}^t(\tau-{\hat\tau})^{-2k-1}\|D_t^ju(\tau)\|_{-n+s+1-j}^2d\tau\leq C\int_{\hat\tau}^t\tau^{-2k}\|Sf(\tau)\|_{s}^2d\tau
\end{gather*}
which proves the assertion.
\end{proof}
%

\section{Proof of propositions and lemmas}
\label{sec:proof:lemma}


\subsection{Proof of lemmas in Section \ref{sec:kakutyo}}

\noindent
Proof of Lemma \ref{lem:kakutyo:a}: By the Taylor formula one can write
\begin{align*}
{f}(z(x, \xi)+{\bar z})=\sum_{|\al|=r}\frac{1}{\al!}z(x, \xi)^{\al}\dif_z^{\al}f({\bar z})
+(r+1)\sum_{|\al|=r+1}\Big[\frac{1}{\al!}z(x, \xi)^{\al}\\
\times \int_0^1(1-\theta)^r\dif_z^{\al} f(\theta z(x, \xi)+{\bar z})
d\theta \Big]
\end{align*}
where $
z(x,\xi)^{\al}\in S(M^{-r},G)$ 
for $|\al|=r$. Since $|z(x, \xi)|\leq CM^{-1}$ the integral belongs to $S(1, G)$ hence  the second term on the right-hand side is in $S(M^{-r-1},G)$ thus the assertion.
\qed
\smallskip

\noindent
Proof of Lemma \ref{lem:eta:seimitu:a}: Let $j\neq d$. Note that
\begin{gather*}
\dif \eta_j/\dif \xi_j
=\chi^{(1)}(M\xi_j\lr{\xi}^{-1})\lr{\xi}^{-1}-M^{-2}\chi^{(1)}(M\xi_j\lr{\xi}^{-1})(M\xi_j\lr{\xi}^{-1})^2\lr{\xi
}^{-1}
\end{gather*}
where $\chi^{(1)}(M\xi_j\lr{\xi}^{-1})(M\xi_j\lr{\xi}^{-1})^2\in S(1, G)$. If $k\neq j$ then
\begin{gather*}
\dif \eta_j/\dif \xi_k=-M^{-1}\chi^{(1)}(M\xi_j\lr{\xi}^{-1})(M\xi_j\lr{\xi}^{-1})(\xi_k\lr{\xi}^{-1})\lr{\xi}^{-1}.
\end{gather*}
Since $\chi^{(1)}(M\xi_j\lr{\xi}^{-1})(M\xi_j\lr{\xi}^{-1})\in S(1, G)$ the assertion is clear.
\qed
\smallskip

\noindent
Proof of Lemma \ref{lem:kihon:1}: Writing $q(t, y,\eta+e_d)={\tilde q}(y,\eta)$ one sees 
\begin{align*}
{\tilde q}(y, \eta)=\sum_{|\al+\be|=2}\frac{1}{\al!\be!}y^{\al}\eta^{\beta}\dif_y^{\al}\dif_{\eta}^{\be}{\tilde q}(0, 0)
+3\sum_{|\al+\be|=3}\Big[\frac{1}{\al!\be!}y^{\al}\eta^{\be}\\
\times \int_0^1(1-\theta)^2\dif_y^{\al}\dif_{\eta}^{\be}{\tilde q}(\theta y, \theta \eta)
d\theta \Big]
\end{align*}
where $\sum_{|\al+\be|=2}y^{\al}\eta^{\be}$ contains no $\eta_d$ because of the Euler's identity. For the case $(a)$ from $\dif_{y_1}^2{\tilde q}(0, 0)=0$ the term $\sum_{|\al+\be|=2}y^{\al}\eta^{\be}$ contains no $y_1$ because ${\tilde q}$ is nonnegative. Therefore $\dif^2_{x_1}{\bar q}\in S(M^{-1}, G)$ and $\dif_{x_j}^2{\bar q}\in S(1, G)$ by Lemma \ref{lem:kakutyo:a}. Now the assertion follows from the Glaeser's inequality. For the case $(b)$ from $
\dif_{\xi_d}{\bar q}=\dif_{\eta_d}{\tilde q}(y,\eta)r
+\sum_{k\neq  d}\dif_{\eta_k}{\tilde q}(y,\eta)r_k$ 
where $r\in S(\lr{\xi}^{-1}, G)$ and $r_k\in S(M^{-1}\lr{\xi}^{-1}, G)$ we have $
\big|\dif_{\xi_d}q \big|\leq CM^{-1/2}\sqrt{q}$. Since   $\dif_{\eta_1}^2{\tilde q}(0, 0)=0$ 
it results $
\big|\dif_{\eta_1}{\tilde 	q}(y, \eta)\big|\leq CM^{-1/2}\sqrt{{\tilde q}(y, \eta)}$ 
which shows $\big|\dif_{\xi_1}q\big|\leq CM^{-1/2}\sqrt{q}$ because  
\begin{gather*}
\dif_{\xi_j}{{\bar q}}=\dif_{\eta_j}{\tilde q}(y,\eta)\big\{\chi^{(1)}(M\xi_j\lr{\xi}^{-1})\lr{\xi}^{-1}+r_{1j}\big\}\\
+\sum_{k\neq j, d}\dif_{\eta_k}{\tilde q}(y,\eta)r_{2 k}+\dif_{\eta_d}{\tilde q}(y,\eta)\dif\eta_d/\dif{\xi_j},\quad j\neq d
\end{gather*}
where $r_{i k}\in S(M^{-1}\lr{\xi}^{-1}, G)$ in view of Lemma \ref{lem:eta:seimitu:a}. 
\qed
\smallskip

\noindent
Proof of Lemma \ref{lem:kihon:2}: Noting that $
|\eta(\xi)+e_d|^2=\sum_{j=1}^{d-1}\eta_j^2+(\eta_d+1)^2=1+k$ with $k\in S(M^{-1}, G)$ we see easily $1/|\eta(\xi)+e_d|=1+{\tilde k}$ with ${\tilde k}\in S(M^{-1}, G)$ hence  $\eta_1(\xi)/|\eta(\xi)+e_d|-\eta_1(\xi)\in S(M^{-1}, G)$. Since ${\psi}(x, \xi)-\eta_1(\xi)/|\eta(\xi)+e_d|\in S(M^{-2}, G)$ by Lemma \ref{lem:kakutyo:a} this together with Lemma \ref{lem:eta:seimitu:a} proves the case $(a)$. The proof for the case $(b)$ is similar.
\qed

\smallskip


\noindent
Proof of Lemma \ref{lem:kokan}:
Write $\ell(t, y,\eta+e_d)={\tilde \ell}(y,\eta)$ then 
\begin{align*}
{\tilde \ell}(y, \eta)=\sum_{|\al+\be|=1}\frac{1}{\al!\be!}y^{\al}\eta^{\beta}\dif_y^{\al}\dif_{\eta}^{\be}{\tilde \ell}(0, 0)
+2\sum_{|\al+\be|=2}\Big[\frac{1}{\al!\be!}y^{\al}\eta^{\be}\\
\times \int_0^1(1-\theta)\dif_y^{\al}\dif_{\eta}^{\be}{\tilde \ell}(\theta y, \theta \eta)
d\theta \Big].
\end{align*}
Since $\sum_{|\al+\be|=1}y^{\al}\eta^{\be}$ contains no $\eta_d$ hence $
\dif_{\eta_d}{\tilde \ell}(y, \eta)\in S(M^{-1}, G)$. Then $
\dif_{\xi_d}{\ell}
\in S(M^{-1}\lr{\xi}^{-1}, G)$. Since $\dif_x^{\al}\psi\in S(M^{-1}, G)$, $(a)$ and $\dif_{\xi}^{\al}\psi\in S(1, G)$, $(b)$ for $|\al|=1$ the  rest of the proof follows from Lemma \ref{lem:kihon:2}.
\qed

\subsection{Proof of Proposition \ref{pro:lam:1} }

Write $z=(x,\xi)$ and $w=(y,\eta)$. Let $g$ be either $\bg$ or $g_{\ep}$ in \eqref{eq:metric:ab}.  
Note that if $\lr{\eta}\leq \lr{\xi}/2\sqrt{2}$ then $|\xi-\eta|\geq (\gamma+|\xi|)/2\geq \lr{\xi}/2$ hence $|\xi-\eta|^4\lr{\eta}^{-2}\geq \gamma \lr{\xi}/2$ and if $\lr{\eta}\geq 2\sqrt{2}\lr{\xi}$ then $|\xi-\eta|\geq (\gamma+|\eta|)/2\geq \lr{\eta}/2$  hence $|\xi-\eta|^4\lr{\eta}^{-2}\geq \gamma\lr{\eta}/16$. Therefore we have
\[
\lr{\xi}/\lr{\eta}+\lr{\eta}/\lr{\xi}\leq C\big(1+\gamma^{-1}\lr{\eta}^{-2}|\xi-\eta|^4\big),\quad \xi,\eta\in\R^d.
\]
Since $g_w(z-w)\geq M^{-2}\lr{\eta}^{-1}|\xi-\eta|^2\geq \gamma^{-1/2}\lr{\eta}^{-1}|\xi-\eta|^2$ for $\gamma\geq M^4$ one has
\begin{equation}
\label{eq:g:sei:1}
\lr{\xi}/\lr{\eta}+\lr{\eta}/\lr{\xi}\leq C\big(1+g_w(z-w)\big)^2
\end{equation}
hence 
\begin{equation}
\label{eq:g:sei:2}
g_z(X)/g_w(X)+g_w(X)/g_z(X)\leq C\big(1+g_w(z-w))^2,\quad 0\neq X\in\R^d\times\R^d
\end{equation}
in particular $g$ is  $\sigma$ temperate  uniformly in $\gamma\geq M^4$  (see \cite[Chapter 18]{Hobook}). Note that \eqref{eq:g:sei:2} implies 
\begin{equation}
\label{eq:g:sei:3}
g_{z+w}(z)\leq C(1+g_w(z))^3.
\end{equation}
In this paper we call a positive $\sigma$, $g$ temperate function (see \cite[Chapter 18]{Hobook}) an admissible weight for $g$.
It is clear from \eqref{eq:g:sei:1} that $\lr{\xi}^s$, $s\in\R$ is 
an admissible weight for $g$.

In this section $A\precsim B$ means that $A\leq CB$ with some $C$ independent of $\lam$, $M$ and $\gamma$ with constraint \eqref{eq:para:seigen}. 

\smallskip

\noindent
Proof of Lemma \ref{lem:b:lam:1}:
 Since ${\bar q}\in S(M^{-2}, G)$ the Glaeser's inequality shows
\begin{equation}
\label{eq:dif:q:one}
|\dif_x^{\al}\dif_{\xi}^{\be}{\bar q}|\precsim \lr{\xi}^{-|\be|}\sqrt{{\bar q}}\,,\quad |\al+\be|=1. 
\end{equation}
Together with  ${\bar b}\geq \lam^{1/2}\lr{\xi}^{-1/2}$ and $\sqrt{{\bar q}}\leq {\bar b}$ this proves that
\begin{equation}
\label{eq:lam:dif:b1}|\dif_x^{\al}\dif_{\xi}^{\be}{\bar b}|\precsim \lambda^{-1/2} \lr{\xi}^{ (|\al|-|\be|)/2}\,{\bar b}\,,\quad |\al+\be|=1.
\end{equation}
Assume \eqref{eq:lam:dif:b1} holds for $1\leq |\al+\be|\leq n$. Since ${\bar b}^2={\bar q}+\lam \lr{\xi}^{-1}$ then for $|\al+\be|\geq n+1\geq 2$ we see
\begin{gather*}
{\bar b}\,\dif_x^{\al}\dif_{\xi}^{\be}{\bar b}=\sum_{|\al'+\be'|\geq 1} C_{\ldots}\dif_x^{\al'}\dif_{\xi}^{\be'}{\bar b}\cdot \dif_x^{\al''}\dif_{\xi}^{\be''}{\bar b}+\dif_x^{\al}\dif_{\xi}^{\be}{\bar q}+\lam\, \dif_x^{\al}\dif_{\xi}^{\be}\lr{\xi}^{-1}.
\end{gather*}
Here note that
\begin{equation}
\label{eq:dif:q:c1}
\begin{split}
|\dif_x^{\al}\dif_{\xi}^{\be}{\bar q}|\precsim M^{-2+|\al+\be|}\lr{\xi}^{-|\be|}\\
\precsim {\bar b}^2\lam^{-1}M^{-2+|\al+\be|}\lr{\xi}^{-(|\al+\be|-2)/2}\lr{\xi}^{(|\al|-|\be|)/2}\\
\precsim {\bar b}^2\lambda^{-1}(M^2\lr{\xi}^{-1})^{(|\al+\be|-2)/2}\lr{\xi}^{(|\al|-|\be|)/2}
\precsim  {\bar b}^2\lambda^{-1}\lr{\xi}^{(|\al|-|\be|)/2}
\end{split}
\end{equation}
since ${\bar b}^2\lam^{-1}\lr{\xi}\geq 1$ and $M^2\lr{\xi}^{-1}\leq 1$. On the other hand we have
\begin{gather*}
|\dif_{\xi}^{\be}\lam \lr{\xi}^{-1}|\precsim \lam \lr{\xi}^{-1-|\be|}\precsim {\bar b}^3\lam^{-1/2}\lr{\xi}^{1/2-|\be|}\\
\precsim {\bar b}^3\lam^{-1/2}\lr{\xi}^{-(|\al+\be|-1)/2}\lr{\xi}^{(|\al|-|\be|)/2}
\precsim {\bar b}^2\lam^{-1/2}\lr{\xi}^{(|\al|-|\be|)/2}
\end{gather*}
for ${\bar b}\leq CM^{-1}$ from which we conclude  \eqref{eq:lam:dif:b1} for any $ |\al+\be|\geq 1$ by induction. 
\qed
\smallskip

\noindent
 Proof of Lemma \ref{lem:b:2}:
Note that
\[
\dif_x^{\al}\dif_{\xi}^{\be}{\bar b}=(\dif_x^{\al}\dif_{\xi}^{\be}{\bar q}+\lam\dif_x^{\al}\dif_{\xi}^{\be} \lr{\xi}^{-1})/2{\bar b}.
\]
Repeating a similar argument proving \eqref{eq:dif:q:c1} we obtain
\begin{gather*}
|\dif_x^{\al+\mu}\dif_{\xi}^{\be+\nu}{\bar q}|\precsim M^{-1}M^{|\mu+\nu|}\lr{\xi}^{-(|\mu+\nu|-1)/2}\lr{\xi}^{-1/2-|\be|}\lr{\xi}^{(|\mu|-|\nu|)/2}\\
\precsim (M^2\lr{\xi}^{-1})^{(|\mu+\nu|-1)/2}\lr{\xi}^{-1/2-|\be|}\lr{\xi}^{(|\mu|-|\nu|)/2}\\
\precsim \lam^{-1/2}\lr{\xi}^{-|\be|}\,{\bar b}\lr{\xi}^{(|\mu|-|\nu|)/2},\quad |\al+\be|=1
\end{gather*}
for $|\mu+\nu|\geq 1$.  This together with \eqref{eq:dif:q:one} shows $\dif_x^{\al}\dif_{\xi}^{\be}{\bar q}/{\bar b}\in S(\lr{\xi}^{-|\be|}, {g})$ for $|\al+\be|=1$. 
On the other hand it is easy to see
\begin{gather*}
|\dif_{\xi}^{\be+\nu}\lam \lr{\xi}^{-1}|\precsim \lam \lr{\xi}^{-1-|\be+\nu|}\precsim {\bar b}^2\lr{\xi}^{-|\be|-|\nu|}
\precsim 
{\bar b}M^{-1}\lr{\xi}^{-|\be|-|\nu|}
\end{gather*}
from which we conclude the assertion. 
\qed

\smallskip

%
\noindent
Proof of Proposition \ref{pro:lam:1}:
Note that $|\dif_x^{\al}\dif_{\xi}^{\be}{\bar b}|\precsim \lr{\xi}^{-|\be|}$ for $|\al+\be|=1$ in view of \eqref{eq:dif:q:one}. Assume $|\eta|\leq c\,\lr{\xi}$ hence
 \begin{equation}
 \label{eq:doto}
 \lr{\xi+s\eta}/C\leq \lr{\xi}\leq C\lr{\xi+s\eta}
 \end{equation}
 where $C$ is independent of  $|s|\leq 1$. 
Thus one has
\begin{gather*}
|{\bar b}(z+w)-{\bar b}(z)|\leq C(|y|+\lr{\xi}^{-1}|\eta|)
\leq C\lr{\xi}^{-1/2}{\bg}_{ z}^{1/2}(w)
\leq C {\bar b}(z){\bg}_{z}^{1/2}(w)
\end{gather*}
hence
\begin{equation}
\label{eq:b:doji}
{\bar b}(z+w)\leq C{\bar b}(z)(1+{\bg}_{z}(w))^{1/2}
\end{equation}
When $|\eta|\geq c\lr{\xi}$ then ${\bg}_{ z}(w)\geq c^2\lr{\xi}$ hence
\[
{\bar b}(z+w)\leq C\leq C{\bar b}(z)\lam^{-1/2}\lr{\xi}^{1/2}\leq C'{\bar b}(z)(1+{\bg}_{z}(w))^{1/2}
\]
thus \eqref{eq:b:doji}. Taking \eqref{eq:g:sei:3} into account we see that ${\bar b}$ is an admissible weight                                                                                                                                                                                                             for ${\bg}$ hence so is $b$. Noting $\lr{\xi}^s\in S(\lr{\xi}^s, {\bg})$ the proof  is completed.
\qed

\smallskip

\noindent
Proof of Lemma \ref{lem:dif:q}:
It is clear from \eqref{eq:dif:q:one} that $
|\dif_x^{\al}\dif_{\xi}^{\be}{\bar q}|\precsim \lr{\xi}^{-1/2}\lr{\xi}^{(|\al|-|\be|)/2}{\bar b}$ for $|\al+\be|=1$. 
For $|\al+\be|\geq 2$ one sees
\begin{gather*}
|\dif_x^{\al}\dif_{\xi}^{\be}{\bar q}|\precsim M^{-2+|\al+\be|}\lr{\xi}^{-|\be|}\\
\precsim \lr{\xi}^{-1}(M^{2}\lr{\xi}^{-1})^{(|\al+\be|-2)/2}\lr{\xi}^{(|\al|-|\be|)/2}
\precsim \lr{\xi}^{-1/2}\,{\bar b}\lr{\xi}^{(|\al|-|\be|)/2}
\end{gather*}
which proves the first assertion. In view of Lemma \ref{lem:eta:seimitu:a} it follows  from the proof of Lemmas \ref{lem:kakutyo:a} and \ref{lem:kihon:1} that   $\dif_{x_1}{\bar q}\in S(M^{-2}, G)$, $(a)$ and $\dif_{\xi_j}{\bar q}\in S(M^{-2}\lr{\xi}^{-1}, G)$, $j=1, d$, $(b)$. Repeating the same arguments proving the first assertion we have
\begin{gather*}
|\dif_x^{\al}\dif_{\xi}^{\be}\dif_{x_1}{\bar q}|\precsim M^{-2+|\al+\be|}\lr{\xi}^{-|\be|}\precsim M^{-1}{\bar b}\lr{\xi}^{(|\al|-|\be|)/2},\;\;(a)\\
|\dif_x^{\al}\dif_{\xi}^{\be}\dif_{\xi_j}{\bar q}|\precsim M^{-2+|\al+\be|}\lr{\xi}^{-1-|\be|}\precsim M^{-1}{\bar b}\lr{\xi}^{-1}\lr{\xi}^{(|\al|-|\be|)/2}, \; j=1, d, \;\;(b)
\end{gather*}
for $|\al+\be|\geq 1$ with together with Lemma \ref{lem:kihon:1} completes the proof.
\qed

\smallskip

 %
\noindent
Proof of Corollary \ref{cor:b:tbibun}: The first assertion is clear from Lemma \ref{lem:b:2}. A repetition of the same arguments proving Lemma \ref{lem:dif:q} shows $\dif_t{\bar q}\in S({\bar b}, \bg)$. Noting
\[
\dif_x^{\al}\dif_{\xi}^{\be}\dif_t{\bar b}=\dif_x^{\al}\dif_{\xi}^{\be}\dif_t{\bar q}/(2{\bar b})-\dif_x^{\al}\dif_{\xi}^{\be}{\bar q}\,\dif_t{\bar q}/(4{\bar b}^{3}),\quad |\al+\be|=1
\]
and ${\bar b}\geq \lr{\xi}^{-1/2}$ the assertion follows from Lemma \ref{lem:dif:q} taking  $\dif_{x_1}\dif_t{\bar q}\in S(M^{-1}, G)$, $(a)$ and $\dif_{\xi_j}\dif_t{\bar q}\in S(M^{-1}\lr{\xi}^{-1}, G)$ for $j=1, d$, $(b)$ into account.
\qed

%

\subsection{Proof of Proposition \ref{pro:dif:omephi:matome}}
\label{sec:ome:phi}

\begin{lem}
\label{lem:hyoka:psi:a} We have $
\dif_x^{\alpha}\dif_{\xi}^{\beta}\psi\in S( \lr{\xi}^{-1/2}M^{-\ep(\al,\be)}\lr{\xi}^{(|\al|-|\be|)/2}, g_{\ep})$  
for $ |\alpha+\beta|\geq 1$. Hence $
\dif_x^{\alpha}\dif_{\xi}^{\beta}\psi\in S( \omega M^{-\ep(\al,\be)}\lr{\xi}^{(|\al|-|\be|)/2}, g_{\ep})$ for $|\al+\be|=1$.
\end{lem}
\begin{proof} Recall that $\psi=\eta_1(\xi)+r$, $(a)$ or $\psi=\varep y_1(x)+cy_d(x)+r$, $(b)$ with $r\in S(M^{-2}, G)$ in view of Lemma \ref{lem:kihon:2}. Let $|\be|\geq 1$ then 
\begin{gather*}
|\dif_{\xi}^{\be}\psi|\precsim M^{-1-\delta_{\ep b}+|\be|}\lr{\xi}^{-|\be|}
\precsim \lr{\xi}^{-1/2}(M^{2\delta_{\ep b}}\lr{\xi})^{-|\be|/2}(M^{2+2\delta_{\ep b}}\lr{\xi}^{-1})^{(|\be|-1)/2}.
\end{gather*}
Let $|\al|\geq 1$ then $|\dif_x^{\al}\psi|\precsim M^{-1-\delta_{\ep a}+|\al|}$ which is bounded by 
\[
\lr{\xi}^{-1/2}(M^{-2\delta_{\ep a}}\lr{\xi})^{|\al|/2}(M^{2+2\delta_{\ep a}}\lr{\xi}^{-1})^{(|\al|-1)/2}.
 \]
 Let $|\al|\geq 1$ and $|\be|\geq 1$, recalling $\ep(\al,\be)=\delta_{\ep a}|\al|+\delta_{\ep b}|\be|\leq |\al+\be|$,  we have
\begin{gather*}
|\dif_x^{\al}\dif_{\xi}^{\be}\psi|\precsim M^{-2+|\al+\be|}\lr{\xi}^{-|\be|}\\
\precsim \lr{\xi}^{-1/2}(M^{-2\delta_{\ep a}}\lr{\xi})^{|\al|/2}(M^{2\delta_{\ep b}}\lr{\xi})^{-|\be|/2}M^{2|\al+\be|-2}\lr{\xi}^{-(|\al+\be|-1)/2}\\
\precsim \lr{\xi}^{-1/2}(M^{-2\delta_{\ep a}}\lr{\xi})^{|\al|/2}(M^{2\delta_{\ep b}}\lr{\xi})^{-|\be|/2}(M^4\lr{\xi}^{-1})^{(|\al+\be|-1)/2}.
\end{gather*}
Since $M^{4}\lr{\xi}^{-1}\leq 1$ by \eqref{eq:seigen} the assertion follows. The second assertion is clear because $\lr{\xi}^{-1/2}
\leq \omega$. 
\end{proof}
\begin{lem}
\label{lem:hyoka:ome:a} We have $
\dif_x^{\alpha}\dif_{\xi}^{\beta}\omega^s\in S( \omega^{s-1} \lr{\xi}^{-1/2}M^{-\ep(\al,\be)}\lr{\xi}^{(|\al|-|\be|)/2}, g_{\ep})$ for $|\al+\be|\geq 1$.
In particular $\omega^s\in S(\omega^s, g_{\ep})$. 
\end{lem}
\begin{proof}First show the assertion for $s=2$.  Since $\omega^2=(t-\psi)^2+\lr{\xi}^{-1}$ noting $\omega\lr{\xi}^{1/2}\geq 1$ 
one sees   for $|\be|\geq 1$
\begin{gather*}
\big|\dif_{\xi}^{\be}\omega^2\big|\precsim \omega M^{-1-\delta_{\ep b}+|\be|}\lr{\xi}^{-|\be|}+M^{-2-2\delta_{\ep b}+|\be|}\lr{\xi}^{-|\be|}+\lr{\xi}^{-1-|\be|}\\
\precsim \omega\lr{\xi}^{-1/2}( M^{2\delta_{\ep b}} \lr{\xi})^{-|\be|/2}(M^{2+2\delta_{\ep b}}\lr{\xi}^{-1})^{(|\be|-1)/2}\\
+ \omega\lr{\xi}^{-1/2}(M^{2\delta_{\ep b}}  \lr{\xi})^{-|\be|/2}(M^{2+2\delta_{\ep b}}\lr{\xi}^{-1})^{(|\be|-2)/2}
+\omega\lr{\xi}^{-1/2} \lr{\xi}^{-|\be|}
\end{gather*}
where  the second term $M^{-2-2\delta_{\ep b}+|\be|}\lr{\xi}^{-|\be|}$ on the right-hand side  is absent when $|\be|=1$.
Let $|\al|\geq 1$ then we see
\begin{gather*}
\big|\dif_x^{\al}\omega^2\big|\precsim \omega M^{-1-\delta_{\ep a}+|\al|}+M^{-2-2\delta_{\ep a}+|\al|}
\\
\precsim \omega \lr{\xi}^{-1/2}(M^{-2\delta_{\ep a}}\lr{\xi})^{|\al|/2}(M^{2+2\delta_{\ep a}}\lr{\xi}^{-1})^{(|\al|-1)/2}\\
+ \omega \lr{\xi}^{-1/2}(M^{-2\delta_{\ep a}}\lr{\xi})^{|\al|/2}(M^{2+2\delta_{\ep a}}\lr{\xi}^{-1})^{(|\al|-2)/2}
\end{gather*}
where if $|\al|=1$ then the term $M^{-2-2\delta_{\ep a}+|\al|}$ on the right-hand side  is absent. Let $|\al|\geq 1$ and $|\be|\geq 1$. Then one sees that
\begin{gather*}
\big|\dif_x^{\al}\dif_{\xi}^{\be}\omega^2\big|\precsim |\omega \dif_x^{\al}\dif_{\xi}^{\be}r|+M^{-4+|\al+\be|}\lr{\xi}^{-|\be|}\\
\precsim \omega M^{-2+|\al+\be|}\lr{\xi}^{-|\be|}+\omega M^{-4+|\al+\be|}\lr{\xi}^{1/2-|\be|}\\
\precsim \omega\lr{\xi}^{-1/2}(M^{-2\delta_{\ep a}}\lr{\xi})^{|\al|/2}(M^{2\delta_{\ep b}}\lr{\xi})^{-|\be|/2}(M^{4}\lr{\xi}^{-1})^{(|\al+\be|-1)/2}\\
+ \omega\lr{\xi}^{-1/2}(M^{-2\delta_{\ep a}}\lr{\xi})^{|\al|/2}(M^{2\delta_{\ep b}}\lr{\xi})^{-|\be|/2}(M^{4}\lr{\xi}^{-1})^{(|\al+\be|-2)/2}.
\end{gather*}
Since $M^{4}\lr{\xi}^{-1}\leq 1$ we have the first assertion for $s=2$.  Since $\lr{\xi}^{-1/2}\leq \omega$ it is clear that $\omega^2\in S(\omega^2, g_{\ep})$ from which it is easy to see $\omega^s\in S(\omega^s, g_{\ep})$ for any $s\in \R$. 
\end{proof}
\begin{lem}
\label{lem:dif:Phi}We have $\phi
\in S(\phi, g_{\ep})$.
\end{lem}
\begin{proof}
Let $|\alpha+\beta|=1$ and write
\begin{equation}
\label{eq:Phi:bunkai}
\dif_x^{\alpha}\dif_{\xi}^{\beta}\phi=\frac{-\dif_x^{\alpha}\dif_{\xi}^{\beta}\psi}{\omega}\phi+\frac{\dif_x^{\alpha}\dif_{\xi}^{\beta}\lr{\xi}^{-1}}{2\omega}=\phi_{\alpha\beta}\phi+\psi_{\alpha\beta}.
\end{equation}
Since $\omega^{-1}\in S(\omega^{-1}, g_{\ep})$ by Lemma \ref{lem:hyoka:ome:a} then
\begin{gather*}
\big|\dif_x^{\mu}\dif_{\xi}^{\nu}\big(\psi_{\alpha\beta}\big)\big|\precsim \omega^{-1}\lr{\xi}^{-1}M^{-\ep(\mu,\nu)}\lr{\xi}^{(|\alpha+\mu|-|\beta+\nu|)/2}\lr{\xi}^{-|\al+\be|/2}
\\
\precsim \phi M^{-\ep(\al+\mu,\be+\nu)}\lr{\xi}^{(|\alpha+\mu|-|\beta+\nu|)/2}
\end{gather*}
in view of $\lr{\xi}^{-|\al+\be|/2}\leq M^{-\ep(\al,\be)}$ and \eqref{eq:phi:uesita}. On the other hand thanks to Lemmas \ref{lem:hyoka:psi:a} and \ref{lem:hyoka:ome:a} it follows that 
\begin{gather*}
|\dif_x^{\mu}\dif_{\xi}^{\nu}\phi_{\alpha\beta}|
\precsim M^{-\ep(\al+\mu,\be+\nu)}\lr{\xi}^{(|\al+\mu|-|\be+\nu|)/2}.
\end{gather*}
Hence using \eqref{eq:Phi:bunkai} the assertion is proved by induction on $|\alpha+\beta|$. 
\end{proof}
\begin{lem}
\label{lem:dif:Phi:seimitu}We have
\begin{gather*}
\dif_x^{\alpha}\dif_{\xi}^{\beta}\phi\in S(\omega^{-1}M^{-\ep(\al,\be)}\lr{\xi}^{-1/2}\lr{\xi}^{(|\al|-|\beta|)/2}\phi, g_{\ep}),\quad |\alpha+\beta|\geq 1.
\end{gather*}
\end{lem}
\begin{proof}
One has $\phi_{\alpha\beta}\in S(\omega^{-1}M^{-\ep(\al,\be)}\lr{\xi}^{-1/2}\lr{\xi}^{(|\al|-|\beta|)/2},g_{\ep})$ for $|\alpha+\beta|\geq1$ by Lemma   \ref{lem:hyoka:psi:a} . From Lemma \ref{lem:hyoka:ome:a}  it follows that
\[
\big|\dif_x^{\mu}\dif_{\xi}^{\nu}\big(\psi_{\alpha\beta}\big)\big|\precsim \omega^{-1}\lr{\xi}^{-1-|\beta|}M^{-\ep(\mu,\nu)}\lr{\xi}^{(|\mu|-|\nu|)/2}
\]
for $|\alpha+\beta|\geq 1$ because $\dif_x^{\al}\dif_{\xi}^{\be}\lr{\xi}^{-1}\in S(\lr{\xi}^{-1-|\be|}, g_{\ep})$ is clear. Since $C\phi\lr{\xi}\geq 1$ and $\lr{\xi}^{-|\be|}\leq M^{-\ep(\al,\be)}\lr{\xi}^{(|\al|-|\be|)/2}$ hence
\[
\psi_{\alpha\beta}\in S(\omega^{-1}M^{-\ep(\al,\be)} \lr{\xi}^{-1/2}\lr{\xi}^{(|\al|-|\beta|)/2}\phi, g_{\ep}),\quad |\alpha+\beta|\geq 1.
\]
Since $\phi\in S(\phi, g_{\ep})$ by Lemma \ref{lem:dif:Phi}  we conclude the assertion from \eqref{eq:Phi:bunkai}. 
\end{proof}
%
%

\noindent
{Proof of Lemma \ref{lem:phi:cho:seimitu}}:
Assume $(a)$. Since $\dif_{\xi_j}\psi\in S(M^{-1}\lr{\xi}^{-1}, G)$ for $j\neq 1$  by Lemma \ref{lem:kihon:2} then the assertion  follows from \eqref{eq:Phi:bunkai}. The assertion for the case $(b)$  is proved similarly.
\qed
%

 
 %
\subsection{Proof of Proposition \ref{pro:omephi:to:g}}
\label{sec:metric}


We start with showing
\begin{lem}
\label{lem:g:conti}There is $C>0$ such that
\[
\omega(z+w)\leq C\omega(z)(1+g_{\ep, z}(w)),\;\; \phi(z+w)\leq C\phi(z)(1+g_{\ep, z}(w))
\]
\end{lem}
\begin{proof} 

First recall that
$ \lr{\xi}^{-1/2}\leq \omega \leq CM^{-1}$.
Assume $|\eta|\geq c\,\lr{\xi}$ hence $g_{\ep, z}(w)\geq c^2 M^{-2}\lr{\xi}\geq c^2M^{-2}\lr{\xi}^{1/2}\lr{\xi}^{1/2}\geq \lr{\xi}^{1/2}$. Therefore
\begin{equation}
\label{eq:ome:kan}
\omega(z+w)\leq CM^{-1}\leq CM^{-1}\lr{\xi}^{1/2}\omega(z)
\leq C\omega(z)(1+g_{\ep, z}(w)).
\end{equation}
Assume $|\eta|\leq c\,\lr{\xi}$. 
 Denote $f=t-\psi$ and $h=\lr{\xi}^{-1/2}$ so that $\omega^2=f^2+h^2$. Note that
\begin{equation}
\label{eq:ome:sa:a}
\begin{split}
|\omega(z+w)-\omega(z)|=|\omega^2(z+w)-\omega^2(z)|/|\omega(z+w)+\omega(z)|\\
\leq 2|f(z+w)-f(z)|+2|h(z+w)-h(z)|
\end{split}
\end{equation}
because $
|f(z+w)+f(z)|/|\omega(z+w)+\omega(z)|$ and $|h(z+w)+h(z)|/|\omega(z+w)+\omega(z)|$ 
are bounded by $2$. It is assumed that constants $C$ may change from line to line but independent of $\gamma\geq M^2\geq 1$. Noting $|f(z+w)-f(z)|=|\psi(z+w)-\psi(z)|$ it follows  from Lemma \ref{lem:hyoka:psi:a} that
\begin{equation}
\label{eq:f:sasa}
\begin{split}
|f(z+w)-f(z)|\leq C (M^{-\delta_{\ep a}}|y|+M^{-\delta_{\ep b}}\lr{\xi+s\eta}^{-1}|\eta|)\\
\leq C \lr{\xi}^{-1/2}(M^{-\delta_{\ep a}}\lr{\xi}^{1/2}|y|+M^{-\delta_{\ep b}}\lr{\xi}^{-1/2}|\eta|)\leq C\omega(z)g_{\ep, z}^{1/2}(w).
\end{split}
\end{equation}
Similarly  we see $
|h(z+w)-h(z)|\leq C\lr{\xi}^{-1}g_{\ep, z}^{1/2}(w)\leq C\lr{\xi}^{-1/2}\omega(z)g_{\ep, z}^{1/2}$. 
Therefore \eqref{eq:ome:sa:a} gives $
|\omega(z+w)-\omega(z)|\leq C\omega(z)g_{\ep, z}^{1/2}(w)$ hence $\omega(z+w)\leq C\omega(z)(1+g_{\ep, z}(w))^{1/2}$.

Turn to $\phi$. If $|\eta|\geq \lr{\xi}/2$ then $g_{\ep, z}(w)\geq M^{-2}\lr{\xi}/4$ hence, taking into account \eqref{eq:phi:uesita} we have
\begin{equation}
\label{eq:kanzou}
\phi(z+w)\leq CM^{-1}\leq CM^{-2}\lr{\xi}\phi(z)\leq C\phi(z)(1+g_{\ep, z}(w)).
\end{equation}
Assume $|\eta|\leq \lr{\xi}/2$ so that \eqref{eq:doto} holds. Note that   $\phi(z+w)-\phi(z)$ is equal to 
\begin{equation}
\label{eq:Phi:sa}
\begin{split}
\frac{(f(z+w)-f(z))(\phi(z+w)+\phi(z))+h^2(z+w)-h^2(z)}{\omega(z+w)+\omega(z)}.
\end{split}
\end{equation}
for $\phi=\omega+f$. From \eqref{eq:f:sasa} it results that $
|f(z+w)-f(z)|\leq C\lr{\xi}^{-1/2}g_{\ep,  z}^{1/2}(w)$. 
It is easy to see that $
|h^2(z+w)-h^2(z)|\leq CM\lr{\xi}^{-3/2	}g_{\ep, z}^{1/2}(w)$. 
Taking these into account \eqref{eq:Phi:sa} yeilds
\begin{equation}
\label{eq:Phi:sa:bis}
\begin{split}
|\phi(z+w)-\phi(z)|\leq C\Big(\frac{\lr{\xi}^{-1/2}}{\omega(z+w)+\omega(z)}(\phi(z+w)+\phi(z))\\
+\frac{M\lr{\xi}^{-3/2}}{\omega(z+w)+\omega(z)}\Big)(1+g_{\ep, z}(w))^{1/2}.
\end{split}
\end{equation}
Since $\phi(z)\geq M\lr{\xi}^{-1}/C$ we have
\begin{align*}
|\phi(z+w)-\phi(z)|\leq C\Big(\frac{\lr{\xi}^{-1/2}}{\omega(z+w)+\omega(z)}(\phi(z+w)+\phi(z))\\
+\frac{\lr{\xi}^{-1/2}}{\omega(z+w)+\omega(z)}\phi(z)\Big)(1+g_{\ep, z}(w))^{1/2}\\
=C(\phi(z+w)+2\phi(z))\frac{\lr{\xi}^{-1/2}}{\omega(z+w)+\omega(z)}(1+g_{\ep, z}(w))^{1/2}.
\end{align*}
If $\lr{\xi}^{-1/2}(1+g_{\ep, z}(w))^{1/2}\big/(\omega(z+w)+\omega(z))<1/3$ then it follows 
\[
\big|\phi(z+w)/\phi(z)-1\big|\leq (\phi(z+w)/\phi(z)+2)/3
\]
from which we have $2\phi(z+w)/5\leq \phi(z)\leq 4\,\phi(z+w)$. If
\begin{equation}
\label{eq:kagi}
\lr{\xi}^{-1/2}(1+g_{\ep, z}(w))^{1/2}\big/(\omega(z+w)+\omega(z))\geq 1/3
\end{equation}
we have, noting $\phi(z)\geq \lr{\xi}^{-1}/(2\omega(z))$, from \eqref{eq:kagi}
\begin{align*}
18(1+g_{\ep, z}(w))\geq 4\lr{\xi}\omega(z+w)\omega(z)\geq \phi(z+w)\big/\phi(z)
\end{align*}
in view of  an obvious inequality $ 2\,\omega(z+w)\geq \phi(z+w)$. Thus  \eqref{eq:kanzou}.
\end{proof}
%
\subsection{Proof of lemmas in Section \ref{sec:pdo:bound}}

\noindent
Proof of Lemma \ref{lem:gyaku}:  
Note that $g_{\ep}\leq {\bar g}={\bar g^{\sigma}}\leq g_{\ep}^{\sigma}$.  In this proof every constant is independent of $\gamma\geq 1$ and $M$. It is clear that $p^{-1}\in S(m^{-1},g_{\ep})$. Write $p\#p^{-1}=1-r$ where $r\in S(M^{-1},g_{\ep})$.  Since 
 \[
 |r|^{(l)}_{S(1,{\bar g})}=\sup_{|\al+\be|\leq l, (x,\xi)\in\R^{2d}}\big|\lr{\xi}^{(|\be|-|\al|)/2}\dif_x^{\al}\dif_{\xi}^{\be}r\big|\leq C_lM^{-1}
 \]
 from the $L^2$-boundedness theorem (see \cite[Theorem 18.6.3]{Hobook}) we have $\|{\rm op}(r)\|\leq CM^{-1}$. Therefore for large $M$ there exists the inverse $(1-\op{r})^{-1}$ in ${\mathcal L}(L^2, L^2)$ which is given by $1+\sum_{\ell=1}^{\infty}r^{\#\ell}\in S(1,{\bar g})$.  (see \cite{Be}, \cite{Ler}, \cite{Ku} ).  Denote $k=\sum_{\ell=1}^{\infty}r^{\#\ell}\in S(1, {\bar g})$ and we will prove $k\in S(M^{-1}, g_{\ep})$. It can be seen from the proof (see, e.g. \cite{Ler}, \cite{Ku} ) that for any $l\in \N$ one can find $C_l>0$, independent of $\gamma$,  such that
\[
|k|^{(l)}_{S(1,{\bar g})}\leq C_l
\]
because $|k|^{(l)}_{S(1,{\bar g})}$ depends only on $l$,  $ |r|^{(l')}_{S(1,{\bar g})}$ with some $l'=l'(l)$ and structure constants of ${\bar g}$ which is independent of $\gamma$. Note that $k$ satisfies $(1-r)\#(1+k)=1$, that is
\begin{equation}
\label{eq:knosiki}
k=r+r\#k.
\end{equation}
Since $r\in S(M^{-1},g_{\ep})$ and $g_{\ep}\leq {\bar g}$ it follows from \eqref{eq:knosiki} that  $
\big|k\big|^{(l)}_{S(1,{\bar g})}\leq C_lM^{-1}$.
Assume that
\begin{equation}
\label{eq:kino}
\sup\big|\lr{\xi}^{(|\be|-|\al|)/2}\dif_x^{\alpha}\dif_{\xi}^{\beta}k\big|\leq C_{\al,\be,\nu}M^{-1-l},\quad \ep(\alpha,\be)\geq l
\end{equation}
for $0\leq l\leq \nu$.
Let $\ep(\alpha, \be)\geq \nu+1$ and note that
\[
\dif_x^{\alpha}\dif_{\xi}^{\beta}k=\dif_x^{\alpha}\dif_{\xi}^{\beta}r+\sum C_{\cdots}\big(\dif_x^{\alpha''}\dif_{\xi}^{\beta''}r\big)\#\big(\dif_x^{\alpha'}\dif_{\xi}^{\beta'}k\big)
\]
where $\alpha'+\alpha''=\alpha$ and $\beta'+\beta''=\beta$. From the assumption \eqref{eq:kino} we have $\dif_x^{\al'}\dif_{\xi}^{\be'}k\in S(M^{-1-\nu}\lr{\xi}^{(|\al'|-|\be'|)/2}, {\bar g})$  if $\ep(\al',\be')\geq \nu+1$ and  if $\ep(\al',\be')\leq \nu$ we have $\dif_x^{\al'}\dif_{\xi}^{\be'}k\in S(M^{-1-\ep(\al',\be')}\lr{\xi}^{(|\al'|-|\be'|)/2}, {\bar g})$. Since $r\in S(M^{-1}, g_{\ep})$ one has
\[
\big(\dif_x^{\alpha''}\dif_{\xi}^{\beta''}r\big)\#\big(\dif_x^{\alpha'}\dif_{\xi}^{\beta'}k\big)\in S(M^{-1-(\nu+1)}\lr{\xi}^{(|\al|-|\be|)/2}, {\bar g})
\]
which implies that \eqref{eq:kino} holds for $0\leq l\leq \nu+1$ and hence for all $\nu$ by induction on $\nu$. 
This  proves that $k\in S(M^{-1},g_{\ep})$.  The proof of the assertions for ${\tilde k}$ is similar.
\qed

\smallskip

\noindent
Proof of Lemma \ref{lem:Ni:hon:1}: One can assume $c=0$. We see that $q(x,\xi)+M^{-1/2}$ is an admissible weight for ${\bar g}$ and $(q+M^{-1/2})^{1/2}\in S((q+M^{-1/2})^{1/2}, {\bar g})$. Moreover $\dif_x^{\al}\dif_{\xi}^{\be}(q+M^{-1/2})^{1/2}\in S(M^{-1/2}\lr{\xi}^{(|\al|-|\be|)/2}, {\bar g})$ for $|\al+\be|=1$. Therefore
\[
q+M^{-1/2}=(q+M^{-1/2})^{1/2}\#(q+M^{-1/2})^{1/2}+r,\quad r\in S(M^{-1}, {\bar g}
)
\]
which proves the assertion.
\qed
\smallskip

\noindent
Proof of Lemma \ref{lem:kihon:fu}:
First note that $m^{\pm 1/2}$ are admissible weights and $m^{\pm 1/2}\in S(m^{\pm1/2},g_{\ep})$. Since $m=m^{1/2}\#m^{1/2}-r$ with $r\in S(M^{-2}m, g_{\ep})$ write
\[
{\tilde r}=(1+k)\#m^{-1/2}\#r\#m^{-1/2}\#(1+{\tilde k})\in S(M^{-1}, g_{\ep})
\]
such that $m^{1/2}\#{\tilde r}\#m^{1/2}=r$.
Therefore one has $
m=m^{1/2}\#(1+{\tilde r})
\#m^{1/2}$ 
and the first assertion follows from Lemma \ref{lem:Ni:hon:2}. 
Write 
\[
{\tilde q}=(1+k)\#m^{-1/2}\#q\#m^{-1/2}\#(1+{\tilde k})\in S(1,g_{\ep})
\]
where $m^{1/2}\#(1+k)\#m^{-1/2}=1$ and $m^{-1/2}\#(1+{\tilde k})\#m^{1/2}=1$ such that
\[
m^{1/2}\#{\tilde q}\#m^{1/2}=q.
\]
Since $k$, ${\tilde k}\in S(M^{-1},g_{\ep})$ one can write ${\tilde q}=q m^{-1}+r$ with $r\in S(M^{-1}, g_{\ep})$. Thanks to Lemma \ref{lem:Ni:hon:2} we have $
\|\op{qm^{-1}}v\|\leq (\sup{\big(|q|/m\big)}+CM^{-1/2})\|v\|$ 
hence
\begin{align*}
\big|(\op{q}u,u)\big|
\leq \big|(\op{qm^{-1}}\op{m^{1/2}}u,\op{m^{1/2}}u)|+ CM^{-1/2}\|\op{m^{1/2}}u\|^2
\end{align*}
proves the second assertion.
\qed

\smallskip

\noindent
Proof of Lemma \ref{lem:m:1:2}
Write  ${\tilde m}_2=m_2\#m_1^{-1}\#(1+k)$ such that $m_2={\tilde m}_2\#m_1$ with $k\in S(M^{-1},g_{\ep})$.  Since ${\tilde m}_2\in S(1, g_{\ep})$ one has 
\[
\|\op{m_2}u\|=\|\op{{\tilde m}_2}\op{m_1}u\|\leq C'\|\op{m_1}u\|
\]
which proves the assertion.
\qed
%

\section{Proof of Proposition \ref{pro:p:to:psi}}
\label{sec:furoku}

\subsection{Geometric characterization of effectively hyperbolic singular points}
\label{sec:geo:char}

In this subsection, for typographical reason, we write $x_0$, $\xi_0$ instead of $t$, $\tau$ respectively and $x=(x_0, x')=(x_0, x_1,\ldots, x_d)$, $\xi=(\xi_0, \xi')=(\xi_0, \xi_1,\ldots, \xi_d)$ so that $
p(x, \xi)=-\xi_0^2+a(x, \xi')$. 
Let $\rho=(0, {\bar \xi})$ be a singular point of $p=0$ and hence ${\bar \xi}_0=0$. We denote $\rho'=(0, {\bar \xi}')$. Consider the Hamilton equation
\[
\frac{d}{ds}\begin{bmatrix}x\\
\xi\end{bmatrix}=J\,\nabla p(x, \xi), \quad \nabla p(x, \xi)=\begin{bmatrix}\dif p(x, \xi)/\dif x\\
\dif p(x, \xi)/\dif \xi\end{bmatrix},\;\; J=\begin{bmatrix}O&I\\
-I&O
\end{bmatrix}
\]
where $I$ is the identity matrix of order $d+1$. We linearize the Hamilton equation at $\rho$. It is clear that the linearization is $dX/ds=J\,\nabla^2p(\rho) X$ with $X={^t}(x, \xi)$ where $\nabla^2p(\rho)$ is the Hesse matrix of $p$ at $\rho$. The coefficient matrix $J\nabla^2p(\rho)$, denoted by $F_p(\rho)$, is called the Hamilton map of $p$ at $\rho$. Therefore denoting the quadratic form defined by the  Hesse matrix by $Q(X, Y)=\olr{X, \nabla^2 p(\rho)Y}$ it is clear that 
\[
Q(X, Y)=\olr{JX, F_p(\rho)Y}={\sigma}(X, F_p(\rho)Y)
\]
because $^{t}\!JJ=I_{2d+2}$ where $\sigma(X, Y)=\olr{JX, Y}$ is the symplectic two form on $V=\R^{d+1}\times\R^{d+1}$. From the definition we see $p(\rho+\ep X)=\ep^2Q(X)/2+O(\ep^3)$ as $\ep\to 0$ and $Q$ has the signature $(r, 1)$ with some $r\geq 0$ because $a(x, \xi')\geq 0$.
Since $a(x, \xi')$ is nonnegative near $\rho'$ the Morse lemma (see, e.g. \cite[Lemma C.6.2]{Hobook} shows that one can find $\phi_1, \ldots, \phi_r$ and $g$ vanishing at $\rho'$, homogeneous of degree $1$, $2$ in $\xi'$ respectively, $C^{\infty}$ in a conic neighborhood of $\rho'$ such that $\nabla\phi_1, \ldots, \nabla\phi_r$ are linearly independent at $\rho'$, $g\geq 0$, $\nabla^2g(\rho')=O$ and 
\begin{equation}
\label{eq:morse}
a(x, \xi')=\sum_{j=1}^{r}\phi_j^2(x, \xi')+g(x, \xi').
\end{equation}
With $\phi_0=\xi_0$ it is clear $Q(X, Y)=-\olr{\nabla\phi_0, X}\olr{\nabla\phi_0, Y}+\sum_{j=1}^r\olr{\nabla\phi_j, X}\olr{\nabla\phi_j, Y}$. Noting $\olr{\nabla\phi_j, X}=\sigma(X, H_{\phi_j})$ where $J\nabla \phi_j=H_{\phi_j}$, it follows that
\begin{gather*}
Q(X, Y)=\sigma(X, F_pY)
=\sigma\big(X, -\sigma(Y, H_{\phi_0})H_{\phi_0}+\sum_{j=1}^r\sigma(Y, H_{\phi_j})H_{\phi_j}\big)
\end{gather*}
and hence $F_pY=-\sigma(Y, H_{\phi_0})H_{\phi_0}+\sum_{j=1}^r\sigma(Y, H_{\phi_j})H_{\phi_j}$. In particular we see
\begin{equation}
\label{eq:F_p:zou}
{\rm Im}F_p={\rm span}\olr{H_{\phi_0}, H_{\phi_1},\ldots, H_{\phi_r}}.
\end{equation}
It is clear that 
\[
{\rm Ker}\,F_p=\{X\in V\mid \sigma(X, H_{\phi_j})=0, j=0,\ldots, r\}=({\rm Im}\,F_p)^{\sigma}.
\]
Note that if $F_pX_{\pm}=\pm\lam X_{\pm}$ with $\lam\neq 0$ then $X_{\pm}\in {\rm Im}\,F_p$ so that $X$ in the proof of Lemma \ref{lem:tokutyo:1} is a linear combination of $H_{\phi_j}$, $j=0, 1,\ldots, r$. Denote by $\Gamma$ the connected component of $\theta=-H_{x_0}=-J\nabla x_0$ in $\{X\in V\mid Q(X)\neq 0\}$ then 
\begin{equation}
\label{eq:hypsui}
\Gamma=\{X=(x, \xi)\mid \xi_0^2>\sum_{j=1}^r\olr{\nabla\phi_j(\rho), X}^2, \, \xi_0>0\}
\end{equation}
which is an open cone in $V$. In what follows for $X\in V$ we denote by $\olr{X}$ the subspace spanned by $X$ and $C=\{X\in V\mid \sigma(X, Y)\leq 0, Y\in \Gamma\}$ and $\Lambda={\rm Ker}\,F_p$.  
Here recall \cite[Corollary 1.4.7]{Ho1}:
\begin{lem}
\label{lem:tokutyo:1}If $F_p(\rho)$ has a nonzero real eigenvalue then  $\Gamma\cap \Lambda^{\sigma}\neq \{0\}$. 
\end{lem}
\begin{proof} Let $\lam\neq 0$ be a real eigenvalue. Show that $-\lam$ is also an eigenvalue of $F_p$. Let $F_pX=\lam X$, $X\neq 0$. Then from $0=\sigma((F_p-\lam)X, Y)=\sigma(X, (-F_p-\lam)Y)$, $Y\in V$ we see that $F_p+\lam$ is not surjective proving that $-\lam$ is also an eigenvalue. Let $F_pX_{\pm}=\pm \lam X_{\pm}$, $X_{\pm}\neq 0$ then  $X_{\pm}\in {\rm Im}F_p=\Lambda^{\sigma}$. Note that the signature of $Q$ is $(r, 1)$ with $r\geq 1$ otherwise $Q(X)$ would be $-\xi_0^2$ and hence $F_p$ has no nonzero eigenvalues. Write $V=V_0\oplus {\rm Ker}F_p$ (direct sum) and consider $Q$ on $V_0$. Since $Q$ is nondegenerate on $V_0$ then $Q$ is of Lorenz signature. We may assume $X_{\pm}\in V_0$. If $\sigma(X_{+}, X_{-})=0$ then, since $\sigma$ is anti-symmetric, $Q$ vanishes on the  $2$ dimensional subspace in $V_0$ spanned by $X_{+}$ and $X_{-}$   which is a contradiction. Thus $\sigma(X_{+}, X_{-})\neq 0$. With $X=\al X_{+}+\be X_{-}\in \Lambda^{\sigma}$ we have
\[
Q(X)=\sigma(\al X_{+}+\be X_{-}, \lam \al X_{+}-\lam \be X_{-})=-2\al\be\lam \sigma(X_{+}, X_{-}).
\]
Then choosing $\al$, $\be$ such that $\al\be\lam \sigma(X_{+}, X_{-})>0$ we conclude either $X$ or $-X$ is in $\Gamma$ .
\end{proof}
\begin{lem}
\label{lem:tokutyo:2}The following three conditions are equivalent;
\begin{description}
\item[\rm(i)] $\Gamma\cap \Lambda^{\sigma}\neq \{0\}$,
\item[\rm(ii)] there is a subspace $H\subset V$ of codimension $1$ such that $
H\cap C=\{ 0\}$ and $\Lambda+\olr{\theta}\subset H$,
\item[\rm(iii)] $\Gamma\cap \Lambda^{\sigma}\cap \olr{\theta}^{\sigma}\neq \{0\}$.
\end{description}
\end{lem}
\begin{proof}${\rm(i)}\Longrightarrow{\rm(ii)}$. First assume $\theta\in \Lambda+\Lambda^{\sigma}$ so  that $\theta=X_1+X_2$ with $X_1\in \Lambda$ and $X_2\in \Lambda^{\sigma}$. Then $0\neq X_2\in \Gamma$ since $\Gamma\cap \Lambda=\emptyset$ and $\Gamma+\Lambda\subset \Gamma$. It is clear that  $\theta\in \olr{X_2}^{\sigma}$ and $\Lambda\subset \olr{X_2}^{\sigma}$. Suppose $\olr{X_2}^{\sigma}\cap C$ contains some $X\neq 0$. Since $\Gamma$ is open then $X_2+Y\in \Gamma$ if $|Y|$ is small hence $\sigma(X_2+Y, X)=\sigma(Y,X)\leq 0$ for $X\in C$ which is a contradiction. Thus $H=\olr{X_2}^{\sigma}$ is a desired subspace. Next consider the case $\theta\not\in \Lambda+\Lambda^{\sigma}$ and hence $(\Lambda+\Lambda^{\sigma})\cap\olr{\theta}=\{0\}$. Take $0\neq Z\in \Gamma\cap\Lambda^{\sigma}$ then recalling $\Gamma$ is open we have
\begin{equation}
\label{eq:ni:1}
\Lambda\subset\olr{Z}^{\sigma},\quad \olr{Z}^{\sigma}\cap C=\{0\}. 
\end{equation}
Thus denoting $T=\olr{Z}^{\sigma}\cap(\Lambda+\Lambda^{\sigma})$ we see
\begin{equation}
\label{eq:ni:2}
\Lambda\subset T,\quad T\cap C=\{0\}.
\end{equation}
Noting that $C\subset \Lambda^{\sigma}$ for $\Gamma+\Lambda\subset \Gamma$ it follows from \eqref{eq:ni:1} that $\Lambda+\Lambda^{\sigma}\not\subset \olr{Z}^{\sigma}$. This proves that ${\rm dim}\,T={\rm dim}(\Lambda+\Lambda^{\sigma})-1$. Write $V=(\Lambda+\Lambda^{\sigma})\oplus W_1$ and $\theta=Y_1+Y_2$ with $Y_1\in \Lambda+\Lambda^{\sigma}$ and $0\neq Y_2\in W_1$ and $W_1=\olr{Y_2}\oplus W_2$. 
Then $H=T+\olr{\theta}+W_2$ is of codimension $1$. From \eqref{eq:ni:2} and $C\subset \Lambda^{\sigma}$ we see $H\cap C=\{0\}$ and hence $H$ is a desired subspace.

\noindent
${\rm(ii)}\Longrightarrow{\rm(iii)}$. Choose $0\neq Y\in V$ such that $\olr{Y}=H^{\sigma}$ then $\olr{Y}\subset \Lambda^{\sigma}\cap \olr{\theta}^{\sigma}$. Show that $Y$ or $-Y$ belongs to $\Gamma$. If not we would have $\olr{Y}\cap \Gamma=\emptyset$. Then by the Hahn-Banach theorem there is $0\neq Z\in V$ such that $\sigma(Z, X)\leq 0, \forall X\in \Gamma$ and $\sigma(Z, X)\geq 0, \forall X\in\olr{Y}$. This shows that $Z\in C$ and $Z\in \olr{Y}^{\sigma}=H$ which is a contradictin. 

\noindent
${\rm(iii)}\Longrightarrow {\rm (i)}$ is trivial.
\end{proof}
%

\subsection{Proof of Proposition \ref{pro:p:to:psi}}

In this subsection we return to the original notation and write $t$ for $x_0$, $x=(x_1,\ldots, x_d)$ and $\tau$ for $\xi_0$, $\xi=(\xi_1,\ldots, \xi_d)$.
After a suitable linear change of local coordinates $x$ we may assume that ${\bar \xi}=(0, \ldots, 0, 1)=e_d$. We write $\rho=(0, 0, e_d)\in \R^{d+1}\times\R^d$ and $\rho'=(0, e_d)\in \R^d\times\R^d$. Thanks to Lemma \ref{lem:tokutyo:2} one can take  $0\neq X\in \Gamma\cap \Lambda^{\sigma}\cap \olr{\theta}^{\sigma}$.  From $X\in \Lambda^{\sigma}$, in view of \eqref{eq:F_p:zou}, $X$ is a linear combination of $H_{\phi_j}(\rho)$ such that $
X=\sum_{j=1}^r\al_jH_{\phi_j}(\rho)+\al_0H_{\phi_0}(\rho)$. 
Since $X\in \olr{\theta}^{\sigma}$ we have $\al_0=0$. We set
\[
f(t, x, \xi)=\sum_{j=1}^r\al_j\phi_j(t, x, \xi)/|\xi|.
\]
Since $H_{f}(\rho)=X\in \Gamma$, noting \eqref{eq:hypsui}, it is clear that $\dif f/\dif t< 0$ at $\rho$ therefore one can write $
f(t, x, \xi)=e(t, x, \xi)(t-\psi(x, \xi))$ 
where $e(\rho)< 0$. It follows from \eqref{eq:morse} 
\begin{equation}
\label{eq:a:no:sita}
a(t, x, \xi)\geq c\,(t-\psi(x,\xi))^2|\xi|^2
\end{equation}
with some $c>0$. Since $-H_{t-\psi}(\rho)\in \Gamma$ we see from \eqref{eq:hypsui} that
\begin{gather*}
1>\sum_{j=1}^r\olr{\nabla\phi_j(\rho), H_{t-\psi}(\rho)}^2=\sum_{j=1}^r\olr{\nabla\phi_j(\rho), J\nabla(t-\psi)(\rho)}^2
=\sum_{j=1}^r\{\phi_j, \psi\}^2(\rho)
\end{gather*}
from which, taking \eqref{eq:morse} and $\nabla^2g(\rho)=O$ into account, we conclude that
\begin{equation}
\label{eq:toku:a}
\big|\{\psi,\{\psi, a\}\}(\rho)\big|<2.
\end{equation}
The next lemma is well known.
\begin{lem}
\label{lem:a:to:psi}Assume $d\psi\neq 0$ and not proportional to $dx_d$ at $\rho'$. Then one can find a system of local coordinates $x=(x_1,\ldots, x_d)$ such that either $d\psi=d\xi_1$ or $d\psi=dx_1+cdx_d$ with some $c\in\R$ at $\rho'$.
\end{lem}
\begin{proof}Since $\dif_{\xi_d}\psi(\rho')=0$ by the Euler's identity one can write $\psi(x, \xi)=\langle{a', \xi'}\rangle+\langle{b', x'}\rangle+b_dx_d+r(x, \xi)$ where $\xi'=(\xi_1, \ldots, x_{d-1})$ and $r$ vanishes at $\rho'$ of order $2$.
Consider the following change of local coordinates $x$. If $a'=0$ hence $b'\neq 0$  the assertion follows by a linear change of coordinates $x'$. If $a'\neq 0$ one can assume $\langle{a', \xi'}\rangle=\xi_1+\cdots+\xi_k$ renumbering $x_j$, $1\leq j\leq d-1$. Replacing the coordinate $x_d$ by $x_d-\sum_{j=1}^kb_jx_j^2/2$ we can assume $\langle{b, x}\rangle=\sum_{j=k+1}^db_jx_j$. Replacing again the coordinate $x_d$ by $x_d-x_1\sum_{j=k+1}^d b_jx_j$ we can assume $b=0$. Then after a linear change of coordinates $(x_1,\ldots, x_k)$ the assertion is clear.
\end{proof}
In Lemma \ref{lem:a:to:psi} we used coordinates change such that $y=x+q(x)$ where $q(x)$ is a quadratic form in $x$. If we cut $q(x)$ off outside a neighborhood of $x=0$  it is clear that the resulting change of coordinates satisfies the requirements in Proposition \ref{pro:p:to:psi}.    

\medskip
\noindent
Proof of Proposition \ref{pro:p:to:psi}:  If $d\psi=0$ or proportional to $dx_d$ at $\rho'$ it suffices to take $\ell=0$ and $q=a$ because $\dif_{\xi_d}^2a(\rho)=0$ by the Euler's identity. Assume $d\psi(\rho')\neq 0$ and not proportional to $dx_d$. Thanks to Lemma \ref{lem:a:to:psi} we may assume $d\psi=d\xi_1$ or $d\psi=dx_1+cdx_d$. Assume $d\psi=d\xi_1$ at $\rho'$. If $\dif_{x_1}^2a(\rho)=0$ it suffices to take $\ell=0$ and $b=a$. Otherwise thanks to the Malgrange preparation theorem one can write 
\[
a(t, x, \xi)=e(t, x, \xi)\big((x_1-h(t, x', \xi))^2+g(t, x', \xi)\big),\quad x'=(x_2,\ldots, x_d)
\]
 where $e>0$ and $h, g$ are  of homogeneous of degree $0$ vanishing at  
 $\rho$. Choose 
 \[
 \ell(t, x, \xi)=e^{1/2}(t, x, \xi)(x_1-h(t, x', \xi)),\;\;q(t, x, \xi)=e(t, x, \xi)g(t, x', \xi)
 \]
and set $\psi_1(t, x',\xi)=\psi(h(t, x', \xi), x', \xi)$ then $d\psi_1=d\psi$ at $\rho'$. From \eqref{eq:a:no:sita} it follows that
\[
q(t, x, \xi)\geq c(t-\psi_1(t, x',\xi))^2|\xi|^2
\]
with some $c>0$. Since $\dif \psi_1/\dif t=0$ at $\rho'$ one can write
\[
t-\psi_1(t, x',\xi)=e'(t, x', \xi)(t-\psi_2(x', \xi)).
\]
Since $d\psi_2=d\psi_1$ at $\rho'$ then $\{{\psi_2},\{{\psi_2}, q\}\}(\rho)=0$ hence it follows from \eqref{eq:toku:a} that $\{\ell, \psi_2\}^2(\rho)<1$. Thus $\psi_2$ is a desired one.  When $d\psi=dx_1+cdx_d$ the proof is similar.
\qed
%


\end{document}